\documentclass{article}

\usepackage[labelfont=bf]{caption}
\usepackage{PRIMEarxiv}
\usepackage{graphicx}
\usepackage[utf8]{inputenc} 
\usepackage[T1]{fontenc}    
\usepackage{amsmath,amsthm}
\usepackage{amssymb}
\usepackage{hyperref}       
\usepackage{url}            
\usepackage{booktabs}       
\usepackage{amsfonts}       
\usepackage{nicefrac}       
\usepackage{microtype}      
\usepackage{cleveref}       
\usepackage{lipsum}         
\usepackage{graphicx}
\usepackage[round]{natbib}
\setcitestyle{authoryear,open={(},close={)}}
\usepackage{doi}

\usepackage{comment}
\usepackage{bm}
\usepackage{multirow}
\usepackage{setspace}
\usepackage{xcolor}
\usepackage{caption}
\usepackage{subcaption}
\usepackage{algpseudocode}

\usepackage[ruled,linesnumbered]{algorithm2e}
\usepackage{eurosym}
\usepackage{wrapfig}
\usepackage{makecell}

\newtheorem{lemma}{Lemma}

\newtheorem{proposition}{Proposition}
\newtheorem{example}{Example}
\newtheorem{observation}{Observation}

\title{On picking operations in e-commerce warehouses: Insights from the complete-information counterpart}

\author{
  Catherine Lorenz\\
  Chair of Management Science/Operations and Supply Chain Management,\\ University of Passau, Passau, Germany\\
  \texttt{catherine.lorenz@uni-passau.de} \\
   \And
  Alena Otto\\
  Chair of Management Science/Operations and Supply Chain Management,\\ University of Passau, Passau, Germany\\
   \And
  Michel Gendreau\\
  Department of Mathematics and Industrial Engineering and CIRRELT,\\
Polytechnique Montr\'{e}al, Montr\'{e}al, Canada
}

\begin{document}
\maketitle

\begin{abstract}
In the flourishing era of e-commerce, major players process \textit{dynamically} incoming orders in \textit{real-time}, thereby balancing the delicate trade-off between fast- and low-cost deliveries. They already use advanced anticipation techniques, like AI, to predict certain characteristics of future orders. However, at the warehousing level, there are still no unambiguous recommendations on how to integrate \textit{anticipation} with intelligent \textit{online optimization algorithms}, nor an unbiased benchmark to assess the \textit{improvement potential} of advanced anticipation over myopic optimization techniques, as  \textit{optimal online solutions} are usually unavailable. 

In this paper, we compute and analyze \textit{Complete-Information Optimal policy Solutions (CIOSs)}, of an exact \textit{perfect anticipation algorithm} with full knowledge of future customer orders' arrival times and ordered items, for picking operations in 
\textit{picker-to-parts} warehouses. We provide analytical properties of perfect anticipation policies and leverage CIOSs to uncover \textit{decision patterns} that enhance simpler algorithms improving both \textit{makespan} (i.e. costs through pickers' working time) and average order \textit{turnover} (i.e. delivery speed). Using a metric similar to the gap to optimality, we quantify the gains from optimization elements and the improvements remaining for advanced anticipation mechanisms. To compute CIOSs, we design tailored dynamic programming algorithms for the Order Batching, Sequencing, and Routing Problem with Release Times (OBSRP-R), which is the first exact algorithm for OBSRP-R in the literature.

Some of the \textit{actionable advice} from our analysis focuses on the largely overlooked \textit{intervention} - the dynamic adjustment of started batches, and the \textit{strategic relocation} of an idle picker towards future picking locations. The former affected over 60\% of orders in CIOSs, integrates easily into myopic policies requiring moderate technological investment, and demonstrated consistent improvements across various policies and warehouse settings. The latter occurred before 39\%-62\% CIOS orders, and a simple implementation could decrease a myopic policy's gap to optimal turnover by 4.3\% on average (up to 14\% in some cases). Notably, our analysis challenges one of the most controversially and heatedly discussed concepts in the warehousing literature: \textit{strategic waiting}. Our study reveals why the latter resembles an "all-in" gamble and harms both makespan and order turnover when intervention is allowed.
\end{abstract}



\section{Introduction} \label{sec:intro}

Worldwide warehousing operations cost businesses about \euro300 billion annually \citep{Hermannetal2019}. E-commerce, a prevalent sector, presents challenges due to \textit{dynamically arriving} customer orders and demands for both free and fast deliveries \citep{henderson2020}. Meeting these demands is expensive as it requires allocating more resources to handle customer requests individually, thereby sacrificing economies of consolidation and scale in warehousing and logistics. Staying \textit{low-cost, while being fast} is a challenge. Not surprising many express delivery start-ups went bankrupt \citep{meyersohn2023}. Conversely, Amazon has successfully reduced delivery costs and speed, boosting its sales and profits, making it one of the world's largest companies by market capitalization alongside Alphabet, Microsoft, Apple, and Nvidia. 
Amazon's CEO, Andy Jassy, stated: “We’ve re-evaluated every part of our fulfillment network over the last year. We obviously like the results, but don’t think we’ve fully realized all the benefits yet” \citep{drummer2023}. Amazon continues expanding its global same-day delivery services \citep{herrington2024}.

\textit{Order picking} is a key component of same-day order fulfillment. In traditional business models, where fulfillment took several days, order picking was often a downstream activity. Orders were accumulated overnight, and delivery plans were created by scheduling backward from delivery deadlines. Order picking was planned only after setting vehicle schedules, leading to unnecessary costs and high space requirements for staging — the intermediate storage between picking and outbound logistics \citep{rijaletal2023}. In same-day delivery models, this approach must be re-evaluated. One option is to reduce order accumulation periods to a few hours while integrating outbound deliveries and order picking. Alternatively, order accumulation periods can be eliminated, by treating both order picking and outbound deliveries as \textit{dynamic} problems with continuously arriving orders. Growing evidence suggests that the arising uncertainty can be managed efficiently \citep{Bertsimas1993}. 
For a detailed discussion on dynamic outbound deliveries, see \citet{azietal2012} and \cite{vocciaetal2019}. This paper focuses on dynamic order picking in warehouses.

\begin{wrapfigure}{r}{0.53\textwidth}
    \centering
    \includegraphics[width=0.45\textwidth]{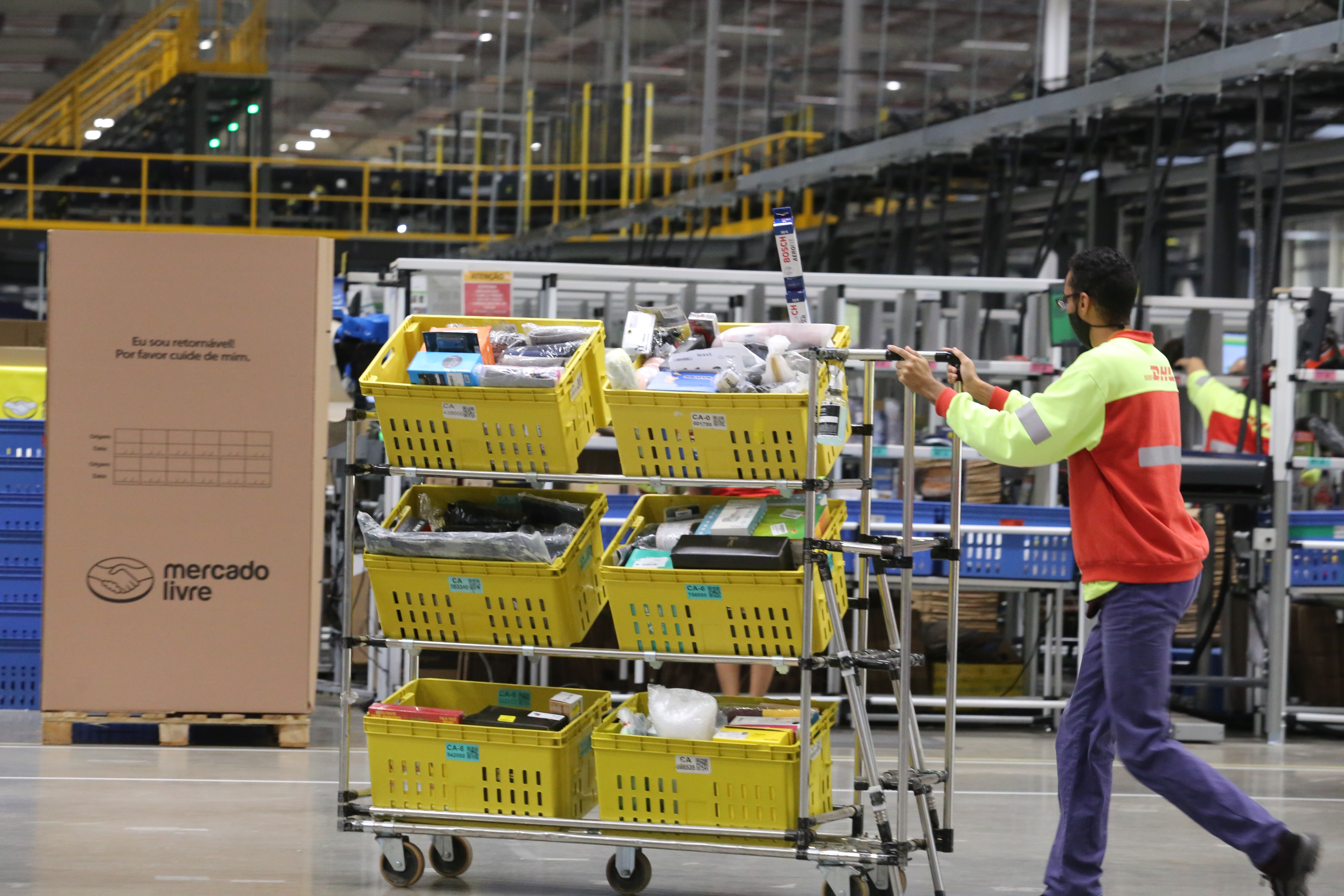}
    \\ \tiny{Source: Governo do Estado de São Paulo. License: CC BY 2.0}
    \caption{Order picking in a picker-to-parts warehouse}
    \label{fig:puscart}
\end{wrapfigure}

A frequently used warehouse design is \textit{picker-to-parts}, where workers travel to collect items (see Figure~\ref{fig:puscart}). Despite advances in automation \citep[see][for a survey]{azadehetal2019}, these warehouses remain popular due to their flexibility in handling demand fluctuations, such as Black Friday \citep{Boysen2019}. Order picking in such warehouses is classified as the \textit{online order batching, sequencing, routing, and waiting problem (OOBSRWP)} \citep{Pardo2023}, which emphasizes dynamically arriving orders and the key decision components: batching (picking multiple orders jointly to reduce costs), sequencing (selecting the next batch), routing (planning the picker's path), and waiting (scheduling pauses to accommodate future orders). 

In this context, seeking \textit{optimal policies} can be highly beneficial -- not for direct implementation, as they are often computationally impractical for real-world instances -- but for revealing decision patterns that can enhance simpler algorithms. 
However, due to the complexity of OOBSRWP, its optimal policies remain unknown, which is common for many dynamic  problems due to the \textit{curse of dimensionality} \citep[see][]{powell2011}. A useful approach to approximate an optimal policy is to compute \textit{complete-information optimal solutions (CIOSs)} of the perfect-information version of OOBSRWP where all order arrival times and characteristics are known in advance. In this paper, we analyze routing, batching, and anticipation strategies employed in CIOSs.

We use this analysis to clarify key questions around effective dynamic policies, including, for example, the configuration of \textit{strategic waiting}. The recent survey of \citet{Pardo2023} describes strategic waiting as the \textit{`by far least studied activity'}, which is expected to have \textit{`a deep influence on the performance of the overall [planning] method [for warehousing operations]'}. Strategic waiting involves intentionally delaying actions to optimize future operations. For instance, if the order queue is short, it may be beneficial to wait for more orders to arrive, allowing for joint picking, which can reduce both completion time and average order turnover. Queueing theory shows that even fully utilized systems have idle times and short queues, providing opportunities for waiting. 
Despite various proposed policies for strategic waiting, the findings remain inconclusive \citep[see][for a summary]{GilBorras2024}.

Another key motivation for designing a perfect-anticipation policy—optimal under perfect information—is to assess the improvement potential of \textit{anticipation algorithms}, such as AI-based methods that predict future orders. Companies like Zalando and Amazon already use AI to forecast and influence customer orders through personalized recommendations and search result sorting \citep{ oneill2024}. 
Emerging paradigms allow direct integration of predictions into optimization algorithms \citep[cf.][]{sadanaetal2024}. To evaluate the improvement potential of these algorithms, we need \textit{perfect anticipation gaps}, which compare them to CIOSs, similar to \textit{optimality gaps} in static optimization. Perfect anticipation gaps allow us to identify scenarios where simple myopic heuristics already approach the best possible outcomes, and where advanced anticipation mechanisms are necessary for substantial improvements.

This paper's analysis relies on insights from perfect-anticipation algorithms with long computational times, as the studied optimization problems are hard-to-solve (NP-hard in the strong sense, see Section~\ref{sec:problem_desc}). Therefore, the focus is narrowed to the most essential aspects:
\begin{itemize}
\item We concentrate on a warehouse's \textit{picking operations}, from order arrival until order delivery to the depot for packaging, highlighting the \textit{sort-while-pick} system, where all items of a customer order are collected in the same bin. 
We analyze two common objectives: \textit{completion time}, or \textit{makespan} minimization (cost-oriented, focused on the picker's time) and the minimization of the \textit{average order turnover time} (service-oriented, focused on fast deliveries). 
For a detailed discussion of warehousing objectives, we refer the reader to \citet{chenetal2010} and \citet{dekosterandbalk2008}.

\item A common warehousing practice designed to avoid congestion and competition among pickers accessing items stored in the same location, is to assign those items to workers during an upstream planning stage \citep[see][]{Schiffer2022}. Also comparative analyses in the literature reveal no differences in the performance rankings of alternative heuristic policies between single-picker and multiple-picker environments 
\citep{GilBorras2024}. Based on this, we focus on the operations of a \textit{single picker}. 
\item Besides traditional class-based storage, where each product item is assigned a location based on its picking frequency, many e-commerce warehouses use a scattered pattern known as mixed-shelves storage, where individual items are placed at multiple locations throughout the warehouse. This approach can reduce unproductive picker walking by increasing the likelihood that ordered items are picked together from nearby locations. Mixed-shelves storage is well-suited to e-commerce because orders  often contain only one or two items \citep{Boysen2019}, the product assortment is extensive, and demand distribution has a long tail \citep{Brynjolfssonetal2003}. The state-of-the-art advice emphasizes the importance of allocating storage space to product items during upstream planning, over the selection of picking locations when planning a picking tour 
\citep[see][]{weidingeretal2018}. Therefore, we assume the \textit{picking positions of ordered items are predetermined}, which allows our analysis to apply to both class-based and mixed-shelves storage.

\end{itemize}

As \textit{our practical contribution}, we analyze CIOSs under two key objectives critical for same-day deliveries: \textit{minimizing costs} (through the picker's working time 
represented by the makespan) and \textit{minimizing delivery speed} (represented by average turnover). We focus on decision patterns in CIOSs that improve both objectives simultaneously. Surprisingly, our detailed analysis of CIOSs downplays the importance of strategic waiting and explains why this may be the case. Instead, it offers actionable insights 
for significantly improving warehousing operations  across both objectives, 
as demonstrated through computational experiments.  While these recommendations may seem intuitive in hindsight, they have been largely overlooked in the literature. 

As \textit{our theoretical contribution}, we design a tailored exact dynamic programming  algorithm (DP) for the perfect-information counterpart of OOBSRWP, which is classified as \textit{the Order Batching, Sequencing, and Routing Problem with Release Times (OBSRP-R)} after \citet{Pardo2023}. To our knowledge, this is the first exact algorithm for OBSRP-R. Computational experiments show that standard mixed-integer programming solvers struggle even with small instances, highlighting the necessity of this tailored DP algorithm. Serving as a perfect-anticipation policy that computes CIOSs, 
our DP sets a benchmark  for emerging AI-based anticipatory approaches. 
We also identify analytical properties of optimal policies for OOBSRWP.

We begin with a literature review (Section~\ref{sec:lit_review}). Sections~\ref{sec:problem_desc} states the problem, Sections~\ref{sec:imp_details} and \ref{sec:analytics_makespan} describe the proposed DP and state analytical properties of online policies, respectively. The main section of this paper is Section~\ref{sec:experiments}, which presents detailed experimental studies and the results. Readers can skip directly to Section~\ref{sec:experiments}, which is written as a stand-alone section. Section~\ref{sec:conclusion} concludes with discussions and an outlook.

\section{Literature review}\label{sec:lit_review}
Over the past decades, warehousing logistics has been a prominent topic in optimization literature. The majority of scientific efforts have concentrated on the picking operation; for a comprehensive overview, we refer to the recent surveys by \citet{Pardo2023, Boysen2019, Vanheusden2022}.

This section outlines three key literature streams relevant to our study: insights from optimal policies for planning operations with dynamically arriving orders; exact solution approaches for OBSRP-R and related problems; the role of advanced anticipation in warehousing operations. 

\subsection{Insights from optimal policies for warehousing with dynamic order arrivals}\label{sec:review_OptP}

Although roughly one in four papers addresses warehousing operations with \textit{dynamically} arriving orders, most develop heuristic solution approaches 
and
evaluate their performance by comparing them with other, often simpler, benchmark strategies 
\citep[see][for an overview]{Pardo2023}. 
These studies show that optimal picker routing and optimal batching of available orders yield significant improvements over simpler heuristics  \citep{Henn2012, GilBorras2024}. However, optimality gaps and the potential for further improvement, e.g., by advanced anticipation, remain unclear. 

A few studies derived \textit{analytical} performance guarantees for online policies. For the batching and sequencing subproblem of OOBSRWP, \citet{Henn2012} and \citet{Alipour2018} showed that a \textit{reoptimization} policy with threshold-based \textit{waiting intervals} 
is at most 2-competitive, meaning the resulting makespan is no more than twice that of the CIOS.
\citet{Lorenza} proved that \textit{immediate myopic reoptimization} (Reopt) is an optimal policy for OOBSRWP under the makespan objective and generic stochastic conditions, based on asymptotic competitive analysis. This means that for \textit{sufficiently large} instances, the result of Reopt coincides with that of a perfect anticipation exact algorithm with a probability of one, and cannot be improved. 

The significance of strategic waiting was extensively studied for  \textit{first-come-first-served (FIFO)} order batching policies using queuing theory analysis \citep{LeDuc2007,Vannieuwenhuyse2009}. In FIFO-batching,  orders are grouped as they arrive, and 
are released as a batch in fixed time intervals or when a fixed number of orders are queued. 
In the context of FIFO-batching, 
strategic waiting increases the likelihood of forming larger batches, realizing savings from picking multiple orders together when the queue is short. 

To gain further insights into strategic waiting, \citet{Bukchin2012} developed an optimal waiting policy to minimize order tardiness and overtime costs for pickers, by analyzing the underlying Markov decision process, for orders arriving according to a Poisson point process. Due to the inherent complexity of OOBSRWP, they significantly simplified batching and routing decisions by assuming that the picking time of a batch depends only on the \textit{number} of constituent orders. 
While the chosen objective function (overtime and tardiness) 
reduces opportunity costs from waiting, the optimal policy still recommended little to no waiting, except when order lead times 
were long compared to the picking time of an order. The authors suggest this may be due to the low order arrival rates used in the computational experiments.

Since then, a variety of approaches for implementing strategic waiting have emerged, ranging from fixed time windows to complex formulas \citep[see][for an overview]{GilBorras2024}. For example, drawing on the batching machine problem,  \citet{Henn2012} evaluated several methods to schedule strategic waiting for the makespan objective. However, his experiments showed that no-wait policies performed better than those involving waiting. Overall, for more advanced and better-performing policies than FIFO-batching, the advice on strategic waiting remains mixed \citep{GilBorras2024, Henn2012}. 

Unlike the studies mentioned above, this paper derives insights on well-performing policies by analyzing an optimal policy for the perfect-anticipation counterpart of OOBSRWP.

\subsection{Exact solution approaches for OBSRP-R and related problems}\label{sec:review_OBSRPR}

OBSRP-R is a perfect-information counterpart of the online problem OOBSRWP. To the best of our knowledge, no solution approach has been developed specifically for OBSRP-R so far. Some papers provide mixed integer programs (MIPs) for OBSRP-R or its subproblems \citep[e.g., ][]{Henn2012,Alipour2018, Cals2021, GilBorras2024}.
Since these MIPs 
were unable to solve even small instances, 
we propose new MIPs in Appendix ~\ref{sec:MIPs}. 

Even without release times, no paper has proposed an exact solution approach that simultaneously addresses \textit{all} related picking decisions -- order batching, sequencing, and picker routing -- due to the problem's complexity. Adding release times increases this complexity. For example, the picker routing subproblem, which is polynomially solvable \citep{Ratliff1983, Schiffer2022}, becomes NP-hard with release times \citep{Bock2024}.  
\citet{Gademann2005}, \citet{Muter2015}, and \citet{Wahlen2023} developed exact branch-and-price algorithms for the Order Batching Problem (OBP), focusing on minimizing the single picker's travel distance.   The current state-of-the-art approach for OBP by \citet{Wahlen2023} works for any monotone routing policy, including \textit{optimal routing}. \citet{Valle2017} developed branch-and-cut algorithms based on a non-compact arc-based formulation for the joint order batching and picker routing problem (OBPR), which minimizes total distance traveled. 
Another exact approach for OBPR by \citet{Schiffer2022} minimizes the  distance in a multi-block warehouse layout with multiple depots. 
Note that the integration of release times would 
disrupt the key properties of the presented solution approaches, rendering them inapplicable to OBSRP-R.



\subsection{The role of advanced anticipation in warehousing operations }\label{sec:review_anticipation}
Recent advancements in predicting customer ordering behavior have leveraged new data sources, such as weather data, click data and page views, consumer reviews, and social media activities \citep[cf.][]{cuietal2018, huang2014, steinkeretal2017}. Progress is moving towards hourly, item-level forecasts of arriving orders \citep{daietal2022, hamdanetal2023}. 

The key to integrating anticipation into planning is to account for both immediate effects of decisions and their delayed consequences (cost-to-go). For instance, when deciding whether to dispatch a worker with a single order, the cost-to-go considers the probabilities and timings of future order arrivals, potential savings from larger batches, and quantifies the long-term impact of this decision. Estimating cost-to-go typically involves sampling and statistical aggregation techniques, including machine learning (ML) and deep learning methods \citep{berksetas2020}. 
Additionally, rather than optimizing based on given forecasts, these forecasts can be directly integrated into optimization. 
By simultaneously considering forecasting and optimization, the framework penalizes forecast errors that affect decisions while tolerating irrelevant ones \citep[see][for a discussion]{sadanaetal2024}.


A few studies have integrated ML-based anticipation methods into warehouse picking operations. \citet{DHaen2022} explored OOBSRWP with fixed delivery truck schedules, aiming to minimize average order tardiness. They used historical data to introduce dummy orders with due dates aligned to truck departures, serving as placeholders for potential future order arrivals. 
\citet{Cals2021} used deep reinforcement learning (DRL) to optimize e-commerce order batching in a combined picker-to-parts and parts-to-picker system.
\citet{shelkeetal2021} proposed an end-to-end DRL method for scenarios where anticipated orders can be picked during one of the previous night shifts.


\section{Problem statement}\label{sec:problem_desc}

In this section, we first present the details of the perfect-information problem formulation -- \textit{the order batching, sequencing and routing problem with release times (OBSRP-R)} --, followed by the online variant with dynamically arriving orders, referred to as OOBSRWP. We denote ${1,\ldots, n}, \forall n \in \mathbb{N}$ as $[n]$. The notation is summarized in Table~\ref{tab:notation}.

\subsection{Warehouse topology}\label{sec:topology}
\begin{figure}
\centering
\includegraphics[scale=0.45]{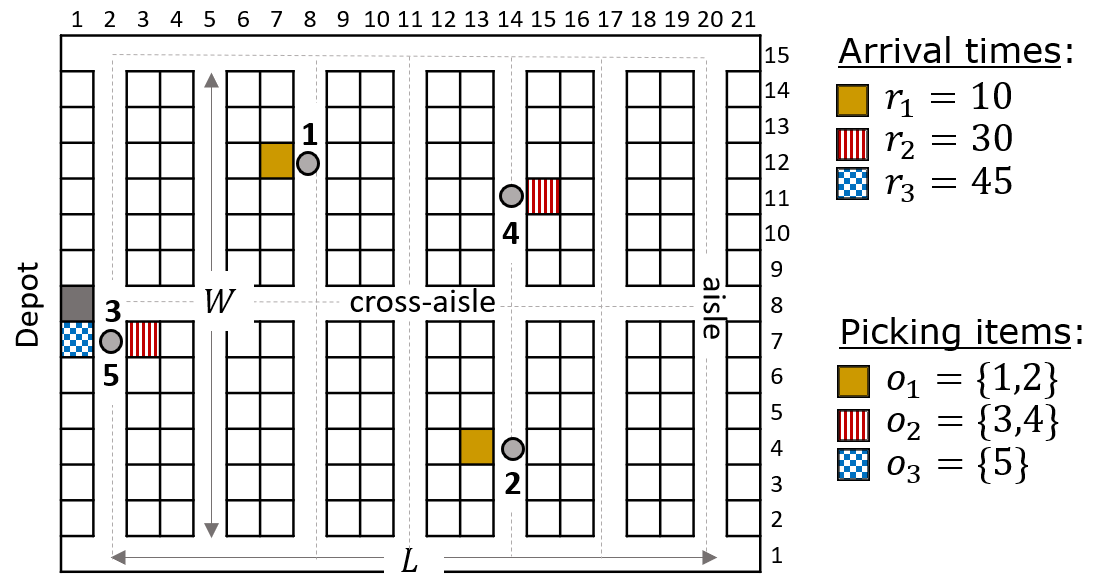}
\caption{Illustrative example of OBSRP-R} 
\scriptsize{\textit{Note. } A warehouse  with $b=3$ cross-aisles and $a=7$ aisles. The batching capacity is $c=2$, a picker speed is $v=1$ cell width per time unit, and a picking time equals $t^p=5$ time units per item. 
The instance has $n^o=3$ orders, the items of which are stored in the colored shelves. The grey circles in the adjacent aisles mark the picking locations of these items.} \label{fig:warehouse}
\end{figure}

OBSRP-R models order picking by a single picker equipped with a pushcart. For simplicity, we refer to the picker's area as a \textit{warehouse}.  The warehouse is rectangular, consisting of $a\geq 1$ vertical \textit{aisles} of length $W$, $b\geq 2$ horizontal \textit{cross-aisles} of length $L$, and a depot $l_d$ located at an arbitrary point (see Figure~\ref{fig:warehouse}). The access point from which an ordered item is retrieved is called the item's \textit{picking location}. 
Each ordered item is linked to a \textit{unique} picking location upon entering the system, so the terms \textit{item} and \textit{(picking) location} are used interchangeably. Note that the picker can only retrieve items from the aisles, not from the cross-aisles.

\subsection{Picker's operations}\label{sec:operations}

The picker starts her operation at the depot and moves through the warehouse at a constant \textit{speed}, $v$. She uses a cart with $c$ bins, as shown in Figure~\ref{fig:puscart}, with each bin dedicated to a separate order, ensuring that items from the same order are grouped in one bin, allowing for sorting while picking (a sort-while-pick system). Note that items from the same order cannot be split across different batches. 
Up to $c$ orders can be picked \textit{simultaneously} (in one \textit{batch}) improving efficiency. We denote $c$ as the \textit{batching capacity}. A batch is completed once all items are picked, and the picker returns the cart to the depot for unloading, after which a new cart is retrieved for the next batch. The retrieval of each item requires fixed \textit{pick time} $t^{p}$.

Let $d(s_1,s_2)$ represent the \textit{shortest walking distance} between any two locations, $s_1$ and $s_2$, following a rectilinear path through the aisles and cross-aisles. The \textit{triangular inequalities} are satisfied, meaning for any three locations $i, l, s$ (picking locations or depot), $d(i,s) \leq d(i,l) + d(l,s)$.

\subsection{The arrival of orders}\label{sec:orders}
OBSRP-R models a perfect-information scenario where all incoming orders, $n^o \in \mathbb{N}$, are known in advance.  Each order $o_j$ 
has a release time $r_j\in \mathbb{R^+}$ and consists of a set of items $o_j:=\{s_j^1,...,s_j^{l}\}, l\in \mathbb{N}$. Orders are indexed by their arrival times, such that $r_1\leq r_2 \leq ...\leq r_{n^o}$. For simplicity, the release time of an item $s$ in order $o_j$ is denoted as $r(s):=r_j$. Each order can have a different number of items. 

The \textit{number} and the \textit{set} of ordered items are denoted as $n^i:=\sum\limits_{j=1}^{n^o} \vert o_j \vert $ and  $S:=\bigcup_{j\in [n^o]} \{o_j\}$, respectively.

\subsection{Feasible solutions}\label{sec:solutions}
A feasible solution of OBSRP-R involves \textit{partitioning orders into batches} for joint collection, \textit{sequencing these batches}, and deciding the \textit{sequence and schedule of picking items} within each batch. 
Let $C(s, \sigma)$ denote the \textit{completion time} of an  item $s\in S$ in a solution $\sigma$, which is the time the item is placed in the cart. For simplicity, we will drop the reference to $\sigma$ when it is clear from the context. Then, a \textit{feasible solution} $\sigma$ of an OBSRP-R instance formally consists of:
\begin{itemize}
\item the picking sequence of the ordered items $\pi$, which is constructed as $\pi:=(\pi^{B_1}, \pi^{B_2}, \ldots, \pi^{B_f})$, $|\pi|=n^i$, with following components:
\begin{itemize}
\item an ordered sequence of batches  $\pi^{\text{batches}}=(B_1, B_2, \ldots, B_f)$, such that the batches form a mutually disjoint partition of the orders: $\{o_1, ..., o_{n^o}\}=B_1 \cup B_2 \cup\ldots \cup B_f, B_l\cap B_k=\emptyset \ \forall k,l \in \{1,...,f\}$; each batch contains at most $c$ orders and the number of batches $f\in \mathbb{N}$ is a decision variable.
\item for each batch $B_l$, a permutation of the items in the included orders $\pi^{B_l}$. 
\end{itemize}
\item the picking \textit{schedule} specifying the completion time $C(s)$ for each item $s\in S$, such that the routing requirements 
are respected, and no item $s$ is picked before its release time $r(s)$. 
\end{itemize}

\begin{table}[t]
    \centering
    \caption{Notation}
    \scriptsize{
    \begin{tabular}{ll}
    \toprule
    \multicolumn{2}{c}{Parameters of an OBSRP-R instance $I$}\\
    \midrule
       $a$ & Number of aisles in the warehouse \\
       $b$ & Number of cross-aisles in the warehouse \\
       $L$ & Length of a cross-aisle\\
       $W$ & Length of an aisle\\
       $l_d$ & Depot\\
       $d()$ & Distance metric between items and depot\\
       $v $ & Picker speed \\
       $t^{p}$ & Picking time to retrieve one item from its storage location and place it into the cart \\
       $c$ & Batching capacity: the maximum number of orders in a batch \\
       $n^o $ & Number of orders \\
       $o_j=\{s_j^1,...s_j^{l}\} $ & Set of items requested by the  $j^{\text{th}}$ order, $j \in [n^o], l \in \mathbb{N}$  \\
       $r_j$ &  Release time of order $o_j, j\in [n^o]$, $r_1\leq ...\leq r_n$  \\
       $r(s)$ & Release time of an item $s$, equals the release time of its respective order, $r(s):=r_j$ if $s\in o_j$\\ 
      $n^i$ &  Number of items, $n^i:=\sum\limits_{j=1}^{n^o} \vert o_j \vert $  \\ 
      $S$ & Set of all ordered items $S:=\cup_{j\in[n^o]}o_j$ \\
      $\sigma$ &Feasible solution\\
      $\pi:=(\pi^{B_1}, \pi^{B_2}, \ldots, \pi^{B_f})$ & Picking sequence of items; with $\pi^{B_i}$ the picking sequence within a batch $B_i$ \\
      $C(\sigma,s) $ / $C(\sigma,o_j) $ & Completion time of item $s$ / order $o_j$ in solution $\sigma$  ($\sigma$ is dropped if clear from the context)\\
      $z^{\text{makespan}}(\sigma)$ & Makespan (total completion time) of solution $\sigma$ \\
       $z^{\text{turnover}}(\sigma)$ & Average order turnover time of solution $\sigma$ \\
      \midrule
      \multicolumn{2}{c}{Notation used in the DP-approach}\\
      \midrule
       $O$ & Sequence of orders sorted with respect to their release times\\ 
       $\Theta_k=(s,m^o,S^{\text{batch}},O^{\text{pend}})$ & State at stage $k\in [n^i+1]\cup\{0\}$;\\
       &for $k\leq n^i$, $k$ items have been picked at this state; 
       $\Theta_{n^i+1}$ is the terminal state; \\
        & $s\in S\cup l_d$ denotes the last picked item for $k\in [n^i]$; \\
        &  $m^o\in[c]\cup\{0\}$ counts the number of orders in the current batch;\\ & set $S^{\text{batch}}\subseteq S$ accommodates  items of the orders from the current batch which have not been picked yet; \\
        &  $O^{\text{pend}}|O$ is the sequence of pending orders\\
       $X(\Theta_k)$ & Set of feasible transitions from state $\Theta_k$ \\
       $f(\Theta_k,x_k)$ & Transition function for state $\Theta_k$ at stage $k$ and transition $x_k\in X(\Theta_k)$\\
       $g^{\text{cost}}(\Theta_k,x_k) \ / \ g^{\text{comp+}}(\Theta_k,x_k)$ & Immediate cost / immediate forward-looking turnover costs of transition $x_k\in X(\Theta_k)$ at state $\Theta_k$\\
       $\Omega^{*,\text{cost}}(\Theta_k) \ / \ \Omega^{*,\text{comp+}}(\Theta_k)$ & Value /  forward-looking completion value of state $\Theta_k$\\ 
      \bottomrule
    \end{tabular}}
    \vspace{0.2cm}
    \label{tab:notation}
\end{table}

\subsection{Objectives}~\label{sec:obj_fun}
We focus on two key objectives in same-day deliveries: minimizing costs (represented by the makespan, which reflects the picker's working time and wages) and minimizing delivery speed (represented by average turnover). We refer to these objectives simply as \textit{makespan} and \textit{turnover}, respectively.

Makespan is the time when the entire picking operation is completed, i.e., when the picker returns the cart to the depot after all items have been collected. The \textit{makespan} objective is defined as:
\begin{align}
Minimize_\sigma ~
z^{\text{makespan}}(\sigma)=\max_{s\in S} \{C(\sigma,s)+\frac{1}{v}\cdot d(s,l_d)\},\label{eq:obj_pushcart_makespan}
\end{align}

Similarly, the \textit{completion time} $C(\sigma, o_j)$ of order $o_j$ in solution $\sigma$ is the time when all items in the corresponding batch have been picked, and the cart has been brought back to the depot:
\begin{align}
C(\sigma, o_j):= \max_{s \in B(\sigma,o_j)} \{ C(s) + \frac{1}{v} \cdot  d(s,l_d)\} , \label{eq:order_comp}
\end{align}
where $B(\sigma,o_j)\subseteq S$  is the set of items in batch $B_l$ if $o_j\in B_l$ in solution $\sigma$. 
The \textit{turnover} time of an order $o_j$ is the time from its arrival $r_j$ to its completion. The \textit{turnover} objective is defined as :
\begin{align}
Minimize_\sigma ~
z^{\text{turnover}}(\sigma)=& \frac{1}{n^o} \sum\limits_{j \in [n^o]} (C(\sigma,o_j) - r_j\}   
=  \frac{1}{n^o} \cdot \sum\limits_{j \in [n^o]} C(\sigma,o_j) + \frac{1}{n^o} \sum\limits_{j \in n^o} r_j \label{eq:obj_turnover}
\end{align}

Observe that the first term of (\ref{eq:obj_turnover}) represents the average order completion time, and the second term is \textit{constant} and represents the average order arrival time. Thus, minimizing turnover is equivalent to minimizing the sum of order completion times. 

For both objectives, note that an optimal picking schedule is uniquely determined by a given picking sequence $\pi$ and its corresponding batches $\pi^{\text{batches}}$. The picking schedule can be reconstructed by collecting each item at the \textit{earliest possible time} after its release time. 
Appendix~\ref{sec:MIPs} provides MIPs for OBSRP-R.

\subsection{An illustrative example of OBSRP-R}~\label{sec:example}
Table~\ref{tab:Solutions_of_ex} presents optimal solutions for the example shown in Figure~\ref{fig:warehouse} with the batching capacity of $c=2$, a picker speed of $v=1$ cell width per time unit, and a picking time of $t^p=5$ time units per item.   For the \textit{makespan} objective, the first batch contains two first orders, while order $o_3$ is picked separately as a single-order batch. The items of the first batch are picked in the sequence $(1, 2, 4, 3)$. The completion time for item 1 is $C(1)=d(l_d,1)+t^p=15$ and for item 2, $ C(2)=c(1)+d(1,2)+t^p=34$. The picker then moves to item 4, but must wait for its release time ($r_2=42$). The first batch is completed at $C(o_1)=C(o_2)=C(3)+d(3,l_d)=69$, and the total makespan is $C(o_3)=76$.  
\begin{table}
\centering
 \caption{ Optimal solutions for the example from Figure~\ref{fig:warehouse}} 
 \scriptsize{
\begin{tabular}{l|cccc}
\toprule
Solution $\sigma^*_i, \ i\in\{1,2\}$ & $\pi (\sigma^*_i)$ & $\pi^{\text{batches}} (\sigma^*_i)$ & $z^{\text{makespan}}(\sigma^*_i)$ & $z^{\text{turnover}}(\sigma^*_i)$\\
\midrule
$\sigma^*_1$: Optimal solution for minimization of $z^{\text{makespan}}$  & $(1,2,4,3,5)$ &  $(\{o_1,o_2\},\{o_3\})$ & 76 & 37.3 \\
$\sigma^*_2$: Optimal solution for minimization of $z^{\text{turnover}}$ & $(1,2,5,3,4)$ & $(\{o_1\}, \{o_3\}, \{o_2\})$ & 99 & 34.7\\
\midrule[\heavyrulewidth]
 \end{tabular}}
  \label{tab:Solutions_of_ex}
\end{table}

An optimal solution for the \textit{turnover} objective involves picking each order separately, prioritizing small orders:  $o_1$ first, followed by $o_3$, then $o_2$. The resulting average turnover is  $z^{\text{turnover}}(\sigma^*_2)=\frac{1}{3}((50-10)+(57-50)+(99-42))=34.7$.

\subsection{The online problem with dynamically arriving orders: OOBSRWP} \label{sec:online_problem}
In reality, the information about arriving orders is not perfect. For a future order $o_j$, the set of picking items ${s_j^1, ..., s_j^l}$ and the number of items $l$ are only revealed upon the order's arrival. Similarly, the total number of orders $n^o$ is unknown. The release times $r_j$ are, in fact, the dynamically revealed \textit{order arrival times}. We refer to the online variant of the problem as OOBSRWP.

\section{Dynamic program} 
\label{sec:imp_details}

This section presents tailored \textit{dynamic programming algorithms (DPs)} for OBSRP-R: an exact algorithm for the \textit{makespan} objective and a heuristic algorithm for the \textit{turnover} objective. Both algorithms share a similar framework, which is outlined below. 
Section~\ref{sec:OPT_imp} describes the associated state graph, Sections~\ref{sec:cost_labeling}-~\ref{sec:turnover_heuristic} state the associated Bellman equations, 
and Section~\ref{sec:dominancerules} introduces customized dominance relations to accelerate the DP approaches. 
Table~\ref{tab:notation} summarizes the additional notation introduced in this section.
 
\subsection{State graph} \label{sec:OPT_imp}
\begin{figure}
\centering
\includegraphics[scale=0.7]{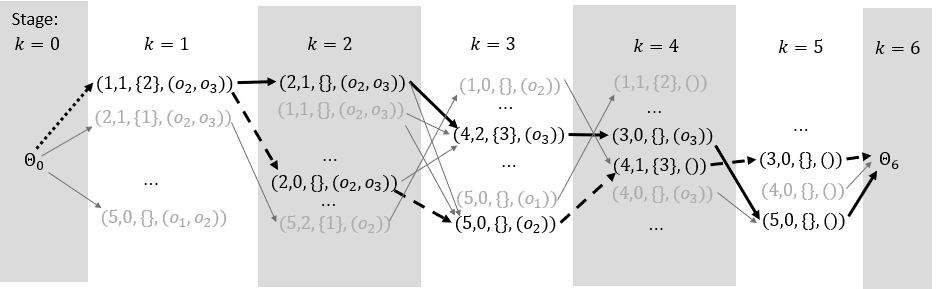}
    \caption{Illustration of the state graph}
    \scriptsize{Note: Extract of the state graph corresponding to the instance 
    from Figure~\ref{fig:warehouse}. The arcs in bold mark the path through the state graph that corresponds to the optimal solution for the makespan objective $\sigma^*_1$; the dashed arcs mark the path of the optimal solution $\sigma^*_2$ for the turnover objective; and the dotted arc belongs to both optimal solutions (see Table~\ref{tab:Solutions_of_ex}). The states involved on at least one of these paths have a black font; some other states are included in light grey, as well as all other feasible transitions between two depicted states, which do not belong to the optimal paths.}
    \label{fig:stategraph}
\end{figure}


Recall that the orders are sorted in non-decreasing order based on their release times. We refer to this sorted sequence as $O$, and for convenience, we abbreviate the phrase \textit{"$O'$ is an ordered subsequence of $O$"} as $O'|O$.

OBSRP-R can be viewed as a sequential optimization problem with stages indexed by  $k\in\{0,1,...,n^i,n^i+1\}$. At stage $k\leq n^i$, the picker has picked $k$ items and, given the current \textit{state} (such as her location, available bins in the cart etc.), faces the \textit{subproblem} to pick the remaining items of the remaining orders. Her immediate \textit{decision} at stage $k$ is about the next item to collect -- and, possibly about the closure or extension of the current batch. Stage $n^i+1$ is reserved for the terminal state as explained below. We depict the resulting sequential problem as a \textit{state graph}. The nodes of this graph are called \textit{states}. The directed edges depict \textit{transitions} between the states of the subsequent stages and are associated with \textit{decisions}. 

\textit{States}. We define state $\Theta_k$ at stage $k\in[n^i]$ as a tuple of the following four variables: 
\begin{align}
\Theta_k\in\{(s,m^o,S^{\text{batch}},O^{\text{pend}})| s\in S, m^o\in\{0,1,\ldots,c\}, S^{\text{batch}}\subseteq S, O^{\text{pend}}|O\}
\end{align}
Variable  $s$ denotes the 
last picked item. Integer $m^o$ counts the number of orders currently assigned to the open batch. The \textit{set} $S^{\text{batch}}\subseteq S$ accommodates the items of the orders from the current batch, which have not been picked yet. Sequence $O^{\text{pend}}|O$ is the \textit{sequence} of pending orders, no items of which have been picked so far.  

In the \textit{initial} state $\Theta_0=(l_d,0,\{\},O)$ at stage $k=0$, the picker is at the depot with an empty cart. The last stage $k=n^i+1$ consists of one state, dubbed \textit{terminal state}. 
It describes the picker's return to the depot with all the orders processed:  $S_{n^i+1}=(l_d,0,\{\},())$. 

States of type $\Theta_k=(s,0,\{\},O^{\text{pend}})$ with $S^{\text{batch}}=\{\}$ and $m^o=0$ 
are called \textit{batch-completion states}. In batch completion states, the picker has completed the batch by picking item $s$ and moved to the depot for unloading, thus in these states, the picker's current position is $l_d$. In all other states, $\Theta_k=(s,m^o,S^{\text{batch}},O^{\text{pend}}), m^o> 0$ , the picker's current position coincides with the picking location of the last picked item $s\in S$.

\textit{Transitions}. We denote a set of feasible decisions, or \textit{feasible transitions}, from state $\Theta_k$ at stage $k\leq n^i$ as $X(\Theta_k)$. The \textit{transition function}  $f(\Theta_k,x_k)=\Theta_{k+1}$ describes the next state after the execution of the decision $x_k\in X(\Theta_k)$ at state $\Theta_k$.

At stage $k=n^i$, the only feasible transitions move to the terminal state from each batch-closure state $\Theta_{n^i}$, other states have no feasible transition. 
At the remaining stages $k=\{0,1,\ldots,n^i-1\}$, transitions $x_k\in X(\Theta_k)$  refer to the selection of the next picking item $s_{k+1}$, and, potentially to the decision of extending- or completing the currently open batch. 
Table~\ref{tab:pcart_transitions} describes feasible transitions for OBSRP-R at stages $k\in \{0,...,n^i-1\}$. 
Lines 1 and 2 describe transitions from a batch-completion state which initiate a new batch. Note that, if the next batch starts by picking an item from a \textit{single-item} order,  the picker has two alternatives: Either to pick this item as part of a larger batch (line 1) or to limit the batch to this one order only (line 2). In lines 3 - 7 and 10,  the currently open batch is extended by a pending order. This is only possible when the batching capacity is not exhausted ($m^o<c$). Thereby, in lines 3 to 6, there are still empty bins in the current batch, but all the items of the already assigned orders have been collected. If an item $s_{k+1}$ from a \textit{single-item} order is selected to be picked next, two alternatives must be distinguished: either the batch is completed by this order (i.e., the next state $\Theta_{k+1}$ is a batch-completion state, see lines 4 and 6), or further orders will be assigned to the current batch (lines 3 and 5). In lines 8-9 and 11-13, at least one item from a commenced order remains in $S^{\text{batch}}$ and is selected to be picked next. Similarly to previous transitions, if only one item remains in $S^{\text{batch}}$ and the batching capacity has not been depleted ($m^o\leq c-1$) (lines 11 and 12), the picker can either close the batch after picking this remaining item and move to a batch-completion state (see line 12), or proceed by extending further the current batch (line 11). 

\begin{table}
\centering
\captionof{table}{Feasible transitions at stages $k=\{0,1,\ldots,n^i-1\}$ }\label{tab:pcart_transitions}\par
\scriptsize{
\begin{tabular}{llll}
\toprule
&Current state & Transition & New state  \\
&$\Theta_k$ & $x_k\in X(\Theta_k)$ & $\Theta_{k+1}=f(\Theta_k,x_k)$    \\
\midrule
\textcolor{gray}{1}&$(s_k, 0, \{\}, O^{\text{pend}})$ & $\forall s_{k+1}\in o_j, o_j\in O^{\text{pend}}$ & $(s_{k+1},1, o_j\setminus\{s_{k+1}\},O^{\text{pend}}\setminus o_j) $ \\
\textcolor{gray}{2}& -\slash\slash- & $\forall s_{k+1}\in o_j, o_j\in O^{\text{pend}}, |o_j|=1$ & $(s_{k+1},0, \{\},O^{\text{pend}}\setminus o_j) $ \\
\textcolor{gray}{3}& $(s_k, m^o, \{\}, O^{\text{pend}}), 0<m^o<c-1$& $\forall s_{k+1}\in o_j, o_j\in O^{\text{pend}}$& $(s_{k+1},m^o+1, o_j\setminus s_{k+1},O^{\text{pend}}\setminus o_j) $ \\
\textcolor{gray}{4}&  -\slash\slash-& $\forall s_{k+1}\in o_j, o_j\in O^{\text{pend}}, |o_j|=1$& $(s_{k+1},0, \{\},O^{\text{pend}}\setminus o_j) $ \\
\textcolor{gray}{5}& $(s_k, m^o, \{\}, O^{\text{pend}}), m^o=c-1$& $\forall s_{k+1}\in o_j, o_j\in O^{\text{pend}}, |o_j|>1$& $(s_{k+1},m^o+1,  o_j\setminus s_{k+1},O^{\text{pend}}\setminus o_j) $ \\
\textcolor{gray}{6}&  -\slash\slash-& $\forall s_{k+1}\in o_j, o_j\in O^{\text{pend}}, |o_j|=1$& $(s_{k+1},0, \{\},O^{\text{pend}}\setminus o_j) $ \\
\textcolor{gray}{7}& $(s_k, m^o, S^{\text{batch}}, O^{\text{pend}}), 0<m^o\leq c-1, |S^{\text{batch}}|>1$  &  $\forall s_{k+1}\in o_j, o_j\in O^{\text{pend}}$ & $(s_{k+1},m^o+1, S^{\text{batch}}\cup o_j\setminus s_{k+1},O^{\text{pend}}\setminus o_j) $\\
\textcolor{gray}{8}& -\slash\slash-  &  $\forall s_{k+1}\in S^{\text{batch}}$ &  $(s_{k+1},m^o, S^{\text{batch}}\setminus\{s_{k+1}\},O^{\text{pend}}) $\\

\textcolor{gray}{9}& $(s_k, m^o, S^{\text{batch}}, O^{\text{pend}}), m^o=c, |S^{\text{batch}}|>1$& $\forall s_{k+1}\in  S^{\text{batch}}$ & $(s_{k+1},m^o, S^{\text{batch}}\setminus\{s_{k+1}\},O^{\text{pend}}) $\\

\textcolor{gray}{10}& $(s_k, m^o, S^{\text{batch}}, O^{\text{pend}}), 0<m^o\leq c-1, |S^{\text{batch}}|=1$& $\forall s_{k+1}\in o_j, o_j\in O^{\text{pend}}$ & $(s_{k+1},m^o+1, S^{\text{batch}}\cup o_j\setminus s_{k+1},O^{\text{pend}}\setminus o_j) $\\
\textcolor{gray}{11}& -\slash\slash- &$\forall s_{k+1}\in S^{\text{batch}}$ & $(s_{k+1},m^o, \{\},O^{\text{pend}}) $\\
\textcolor{gray}{12}& -\slash\slash-& $\forall s_{k+1}\in S^{\text{batch}}$ & $(s_{k+1},0, \{\},O^{\text{pend}}) $\\
\textcolor{gray}{13}& $(s_k, m^o, S^{\text{batch}}, O^{\text{pend}}), m^o=c, |S^{\text{batch}}|=1$& $\forall s_{k+1}\in S^{\text{batch}}$ & $(s_{k+1},0, \{\},O^{\text{pend}}) $\\
\midrule[\heavyrulewidth]
  \end{tabular}
}
\end{table}

Appendix~\ref{sec:dp_correct} proves the correctness of the formulated state graph. 
Figure~\ref{fig:stategraph} shows a snippet of the graph for the instance in Figure~\ref{fig:warehouse}, highlighting the paths of the optimal solutions from Table~\ref{tab:Solutions_of_ex}.

\subsubsection{Bellman equations for the makespan objective}\label{sec:cost_labeling}

The immediate transition cost $g^{\text{cost}}(\Theta_{k},x_{k})$ from a state $\Theta_{k}$ by taking decision $x_k$ which results in a state $\Theta_{k+1}$ (with last-picked item $s_{k+1}$) represents the \textit{time} required to make this transition. It is the maximum of the picker's walking time from either item $s_k$ of the previous state -- or the depot $l_d$ -- depending on the nature of the previous state, to the next state's item $s_{k+1}$; and the waiting time for the release of $s_{k+1}$, plus the picking time $t^p$. When $\Theta_{k+1}$ is a batch-completion state, where the picker completes a batch by returning her cart to the depot, the walking time from $s_{k+1}$ to $l_d$ is added. 
See Appendix~\ref{sec:details_bellman} for precise formulas.

The minimal makespan, which is the time of the shortest path in the defined state graph from the initial state to the terminal state, is computed recursively in a forward manner  with the following Bellman equations:
\begin{align}
\Omega^{*,\text{cost}}(\Theta_{k+1})=
    \min_{(\Theta_{k}, x_{k})\in f^{-1}(\Theta_{k+1})} \{ \Omega^{*,\text{cost}}(\Theta_{k}) +  g^{\text{cost}}(\Theta_{k},x_{k})  \} \label{eq:Bellman_makesp}
\end{align}
where $\Omega^{*,\text{cost}}(\Theta_{k})$ is the value of state $\Theta_{k}$ and the inverse image $f^{-1}(\Theta_{k+1})$  represents the set of all feasible transitions (see Table~\ref{tab:pcart_transitions}) to reach state $\Theta_{k+1}$:  $f^{-1}(\Theta_{k+1})=\{ (\Theta_{k},x_{k}) \ \vert \ \Theta_{k} \text{ is a state at stage } {k}, x_{k}\in X(\Theta_{k}), \text{ and } f(\Theta_{k},x_{k})=\Theta_{k+1} \}$. The value of the initial state $\Omega^{*,\text{cost}}(\Theta_0):=0$. 
Finally, the minimal makespan for the given OBSRP-R instance equals the value of the terminal state: $z^{*,\text{makespan}}=\Omega^{*,\text{cost}} (\Theta_{n^i+1})$.
\subsubsection{Bellman equations and heuristic labeling for the turnover objective} \label{sec:turnover_heuristic}
For the average turnover objective, \textit{w.l.o.g} we compute the sum of order completion times in the state graph (see the discussion in Section~\ref{sec:obj_fun}).  

An exact algorithm for the turnover objective requires computing \textit{labels} in each state $\Theta_{k}$ with two values, on one hand the clock time $\Omega^{\text{cost}}(\Theta_{k})$, computed similarly to (\ref{eq:Bellman_makesp}), and on the other hand the associated completion time value, which represents the sum of the completion times for all completed orders. In certain cases, several labels per state may be required to find an optimal solution \citep[see][on labeling algorithms]{Irnich2005}. However,
maintaining all Pareto-optimal labels for each state would drastically increase the size of the state graph, slowing down the approach significantly. To address this, we introduce a heuristic label — the \textit{forward-looking completion value} — for faster and more efficient processing. The resulting DP returns \textit{optimal solutions} when no waiting time is involved. However, in certain cases where waiting for incoming orders is necessary, the DP \textit{may yield suboptimal solutions}, as it slightly overestimates the benefits of batching (see example in Section~\ref{sec:example_labeling} below). 


The \textit{forward-looking completion value}  of a state 
represents the sum of completion times for all \textit{already completed} orders, plus the clock time of the current state summed up for all \textit{not yet completed} orders. Intuitively, a low forward-looking completion value reflects both the short-term benefits of a low average turnover for completed orders and the long-term benefits of a low current clock time.  

We introduce the \textit{immediate forward-looking turnover costs}  $ g^{\text{comp+}}(\Theta_k,x_k)$ of a transition from a state $\Theta_k$ by taking decision $x_k$. It sums up the duration of the transition (i.e. $g^{\text{cost}}(\Theta_k,x_k)$) for every uncompleted order at stage $\Theta_k=(s_k,m^o,S^{\text{batch}},O^{\text{pend}})$:
\begin{align}
g^{\text{comp+}}(\Theta_k,x_k)= g^{\text{cost}}(\Theta_k,x_k)\cdot(m^o+ \vert O^{pend} \vert) \qquad \forall k \in \{0,...,n^i\} \label{eq:immediate_turnover}
\end{align}

Only one label is stored in each state $\Theta_{k+1}$ with the best forward-looking completion value which is computed recursively in a forward manner with the following Bellman equations:
\begin{align}
\Omega^{*,\text{comp+}}(\Theta_{k+1})=
    \min_{(\Theta_k, x_k)\in f^{-1}(\Theta_{k+1})} \{ \Omega^{*,\text{comp+}}(\Theta_k) +  g^{\text{comp+}}(\Theta_k,x_k)  \}, \label{eq:Bellman_compl}
\end{align}
The value of the initial state $\Omega^{*,\text{comp+}}(\Theta_0):=0$. 
Finally, the algorithm outputs $\Omega^{*,\text{comp+}}(\Theta_{n^i+1})$ of the terminal state and computes the average turnover time of the resulting solution straightforwardly with Formula~(\ref{eq:obj_turnover}).

\subsubsection{Illustrative example for the turnover objective} \label{sec:example_labeling}
Consider OBSRP-R instance $I_1$ with  single-item orders $o_1=\{s_1\}$, $o_2=\{s_2\}$ and $o_3=\{s_3\}$, all with release times $r_1=r_2=r_3=0$. The picker moves at unit speed $v=1$, with a batching capacity of $c=2$ and negligible picking time $t^p=0$. The distances are as follows: $d(l_d,s_1)=3, d(l_d,s_2)=5, d(s_1,s_2)=4, d(s_2,s_3)=7$,  and $d(l_d,s_3)=2$. For simplicity,  assume that order $o_1$ has to be picked first, followed by $o_2$, then $o_3$. 
\begin{figure}[h]
    \centering
    \includegraphics[width=\linewidth]{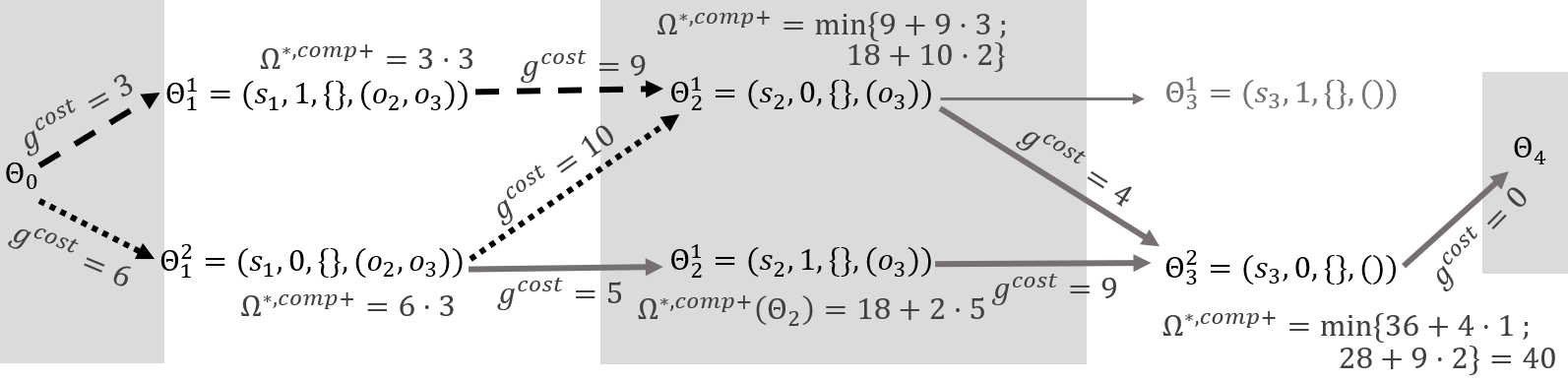}
    \caption{Forward-looking completion values for instance $I_1$}
    \label{fig:example_label}
\end{figure}

Figure~\ref{fig:example_label} illustrates the state-graph and the associated forward-looking completion values for $I_1$. For instance, the first stage contains two states, order $o_1$ can be picked in a batch with multiple orders resulting in $\Theta_1^1$, or as a single-order batch, resulting in batch-completion state $\Theta_2^1$. The transition to the latter has a $g^{\text{cost}}$ of $2\cdot 3=6$ representing the time of the return trip to $s_1$ from the depot; the respective immediate forward-looking turnover costs $g^{\text{comp+}}$ equal $6\cdot 3=18$. We multiply 6 with 3, because it is the completion time of the first order and, as the current clock time, it is part of the completion time of the orders $o_2$ and $o_3$, see Formula (\ref{eq:immediate_turnover}). The resulting forward-looking completion value is $\Omega^{*,\text{comp+}}(\Theta_1^2)=0+18=18$. Similarly, we receive the value $\Omega^{*,\text{comp+}}(\Theta_1^1)=0+3\cdot 3=9$. The state $\Theta_2^1$ of the second stage can be reached from both $\Theta_1^1$ and $\Theta_1^2$ with immediate forward-looking turnover costs $g^{\text{comp+}}$ of $(4+5)\cdot(1+2)=27$ and $5\cdot 2\cdot (0+2)=20$, respectively, and the best forward-looking value of $\Theta_2^1$ consequently equals $\Omega^{*,\text{comp+}}(\Theta_2^1)=\min\{9+27;18+20\}=36$. In other words, the DP recommends picking the two first orders in one batch $\{o_1, o_2\}$ (dashed arcs); the partial solution $\{o_1\}, \{o_2\}$ (dotted arcs) is discarded for further analysis. With a similar logic, as illustrated in Figure~\ref{fig:example_label}, the algorithm proceeds and outputs the solution $\{o_1, o_2\},\{o_3\}$, which is indeed an optimal solution for $I_1$ with the average turnover of $\frac{40}{3}$. 

Now, consider an instance $I_2$ 
with the same parameters as for $I_1$, but release time $r_3=17$. The forward-looking value $\Omega^{*,\text{comp+}}(\Theta_2^1)$ remains 36 and the partial solution  $\{o_1\}, \{o_2\}$ (dotted arcs) is discarded for further analysis. This, however, is an erroneous decision given the \textit{waiting time} before picking the last order $o_3$ in $\Theta_3^2$. The latter, with the algorithm, is reached through $\Theta_2^1$ with a $g^{\text{cost}}$ of 2+3+2, where the first and last summands represent the walking time and the second the waiting time until $r_3$; resulting in an immediate forward-looking turnover cost  $g^{\text{comp+}}$ of $7\cdot 1$ and a value $\Omega^{*,\text{comp+}}(\Theta_3^2)=\min\{36+7;28+9\cdot 2\}=43$. Starting with the partial solution $\{o_1\}, \{o_2\}$ would lead to no waiting time for picking $s_3$, thus immediate costs $g^{\text{cost}}$ and $g^{\text{comp+}}$ of $4$ and $4\cdot 1$ respectively, and the forward-looking completion value of $\Omega^{*,\text{comp+}}(\Theta_3^2)=38+4=42$. This would have provide a slightly shorter average turnover, $\frac{43}{3}-\frac{17}{3}$ compared to $\frac{42}{3}-\frac{17}{3}$ achieved by the algorithm, demonstrating that in some cases, the latter may output slightly sub-optimal solutions.
\subsection{Dominance rules for transitions in the DP formulation } \label{sec:dominancerules}

Certain transitions in the state graph of the DP can be omitted, thereby reducing its complexity. 

For a given instance, a feasible solution $\sigma$ is called \textit{weakly dominated}, if there exists another feasible solution $\sigma'$ with identical or better objective value, i.e. if $z^{\text{makespan}}(\sigma')\leq z^{\text{makespan}}(\sigma)$ or $z^{\text{turnover}}(\sigma')\leq z^{\text{turnover}}(\sigma)$, in case of the makespan- or turnover objective, respectively. 
Similarly, we call a \textit{transition} in the state graph\textit{ weakly dominated}, if \textit{every} feasible solution involving this transition is weakly dominated by another feasible solution not involving this transition. 
By definition, weakly dominated transitions may be omitted in the DP. 

Propositions ~\ref{prop:dom_in_batch_1},~\ref{prop:dom_in_batch_2} and ~\ref{prop:dom_batch-closure}, identify weakly dominated transitions, based on the release time of the next selected picking item. Intuitively, they prohibit selecting items whose release times are too far in the future. The propositions hold for the makespan objective, \textit{and} for the turnover objective. 


\begin{proposition}~\label{prop:dom_in_batch_1}
Consider a current state $\Theta_k=(s_k,m^o,S^{\text{batch}},O^{\text{pend}})$ with $\vert S^{\text{batch}} \vert \geq 1$. Then any transition $x_k\in X(\Theta_k)$ that visits next picking location $s_{k+1}\in o_j, o_j\in O^{\text{pend}}$ is weakly dominated, if: 
\begin{align}
r_j \geq  \Omega^{*,\text{cost}}(\Theta_k) + \frac{2}{v} \cdot (L+W) +t^{p} \qquad \qquad \qquad (D1) \nonumber 
\end{align}
\end{proposition}
\begin{proof}
See Appendix~\ref{sec:proof_dom_in_batch_1}.   
 \end{proof}
Intuitively, Proposition~\ref{prop:dom_in_batch_1} prohibits selecting item $s_{k+1}$ with a late release time, if there are items of a commenced order in the current batch ($\vert S^{\text{batch}} \vert \geq 1$) that can be picked first. Note that $L+W$ is an upper bound for the distance between any two locations.

\begin{proposition}~\label{prop:dom_in_batch_2}
Consider a current state $\Theta_k=(s_k,m^o,\{\},O^{\text{pend}})$ with $m^o>0$. Then any transition $x_k\in X(\Theta_k)$ that visits next picking location $s_{k+1}\in o_j, o_j\in O^{\text{pend}}$ is weakly dominated, if 
\begin{align}
r_j \geq  \Omega^{*,\text{cost}}(\Theta_k) + \frac{2}{v} (L+W)  \qquad \qquad \qquad (D2)
\nonumber
\end{align}
\end{proposition}
\begin{proof}
See Appendix~\ref{sec:proof_dom_in_batch_2}.   
 \end{proof}
Intuitively, Proposition~\ref{prop:dom_in_batch_2} requires completing the current batch before moving to item $s_{k+1}$ with a late enough release time, if all the items in the commenced orders of the current batch have been picked.


Finally, Proposition~\ref{prop:dom_batch-closure} describes the dominance of feasible transitions from batch-completion states. It prevents starting a new batch with a pending order $o_j\in O^{\text{pend}}$ that has a late release time, if we can complete another pending order in a single-order batch first.
\begin{proposition}\label{prop:dom_batch-closure}
Consider a current state $\Theta_k=(s_k,0,\{\},O^{\text{pend}})$. Let $o_{j_{min}} \in O^{\text{pend}} $ be the uncompleted order with the earliest release time in $O^{\text{pend}}$. Then, any transition $x_k\in X(\Theta_k)$ that visits next picking location $s_{k+1}\in o_{j}\neq o_{j_{min}}, o_{j}\in O^{\text{pend}}$ is dominated, if: 
\begin{align}
r_{j}\geq  \max\{r_{j_{min}};\Omega^{*,\text{cost}}(\Theta_k)\}+\frac{1}{v}\cdot (UB+L+W) + \vert o_{j_{min}} \vert \cdot t^p , \qquad (D3) \nonumber  
\end{align}
with $UB:=(2L+(a(o_{j_{min}})+1)W)$, where $a(o_{j_{min}})$ is the number of aisles containing picking items of $o_{j_{min}}$.
\end{proposition}
\begin{proof}
See Appendix~\ref{sec:proof_dom_batch_closure}.
 \end{proof}
Note that $UB(o_{j_{min}})$ presents an upper bound for the walking distance when picking $o_{j_{min}}$ in a single-order batch, see Lemma 1 of \citet{Lorenza}.

The right-hand sides of dominance rules ($D1$), ($D2$), and ($D3$) represent bounds (\textit{thresholds}) that are independent of the particular picking locations, which allows to efficiently integrate them in the DP. Specifically, at each state $\Theta_k, k \in \{0,...,n^i\}$, transitions $x_k\in X(\Theta_k)$ are considered in the non-decreasing order of the release times of associated next picking item $s_{k+1}$. Then, no more transitions from $\Theta_k$ 
must be considered, once the release time $r(s_{k+1})$ of the next item of a transition exceeds the associated threshold.  Algorithm~\ref{alg:in_batch_dom} provides details on the implementation. 

\begin{algorithm}
\scriptsize{
\SetAlgoLined
\LinesNumbered
Let $o_{j_{min}}\in O^{\text{pend}}$ be the first order with the minimum release time in $ O^{\text{pend}}$.\\
\For{all $s_{k+1}\in S^{\text{batch}}$ in arbitrary order}{perform applicable transition(s) of Table~\ref{tab:pcart_transitions} }
\For{ all $o_j\in O^{\text{pend}}$ in the given order of $O^{\text{pend}}$}{
  \uIf{ ($D1$)-($D3$) is True}{
  break and go to line 13\;}
  \Else{
    \For{all $s_{k+1}\in o_j$ in arbitrary order}{
    perform applicable transition(s) of Table~\ref{tab:pcart_transitions}
     }}
  }  
}
\caption{Construction of non-dominated transitions from state $\Theta_k=(s_k,m^o,S^{\text{batch}},O^{\text{pend}})$ } \label{alg:in_batch_dom}
\end{algorithm}

\section{Analytical properties of online policies for the makespan objective}\label{sec:analytics_makespan}
For the makespan objective in OOBSRWP, we can formulate several analytical properties. 

Lemma~\ref{lemma:waitingtime_placement} states that, \textit{ceteris paribus},  for a fixed instance and  optimal policy, the outcome will not worsen if we consolidate all waiting and idle times and place them at the beginning.  After this initial delay, the picker can proceed with forming batches and routing following the original policy, but avoiding further delays.

\begin{lemma}\label{lemma:waitingtime_placement}
Consider a fixed instance $I$ of OOBSRWP with the makespan objective, and any perfect-anticipation solution (CIOS) for this instance with a picking schedule $C^*(s), s\in S$. Let with $w^*(I)$ be the sum of all waiting and idle times in this schedule. Then there exists an alternative optimal solution, with identical makespan, picking sequence and batches, with an adjusted schedule $\tilde{C}^*(s), s\in S$, in which the picker starts the operation at time $w^*(I)$. 
\end{lemma}
\begin{proof}
Consider a counterpart of instance $I$ with zero release times for all orders, $I^{r=0}$. The cumulative waiting and idle time $w^*(I)$ can be reinterpreted as the difference between the perfect-anticipation optimal makespans of these instances: $w^*(I)=z^{*}(I)-z^{*}(I^{r=0})$. The proof immediately follows.
 \end{proof}

\textit{Strategic relocation} is a form of anticipation where the picker moves towards the picking position of an item before its actual arrival, it occurs frequently in the complete-information solutions (CIOSs) (see Section~\ref{sec:CIOS_analysis}). Some studies in the routing literature demonstrated the positive effects of strategic relocation during idle times \citep{Bertsimas1993}.   Lemma~\ref{lemma:strategic_reloc} limits the benefits of anticipatory strategic relocation for the makespan objective. It compares a CIOS with an algorithm ALG that constructs the same routes and the same batches, but without performing strategic relocation, i.e. the picker in ALG remains at its last position until the arrival of the next planned picking item. The improvement potential for ALG is limited by the warehouse's dimensions ($W+L$). 
\begin{lemma}\label{lemma:strategic_reloc}
    For any instance $I$ of OOBSRWP with the \textit{makespan} objective, and the algorithm ALG as defined above, the following holds: 
    \vspace{-0.3cm}
    \begin{align}
    z^{\text{ALG}}(I)\leq z^{*}(I) +  \max_{i,j\in S\cup l_d}\{ d(i,j) \}
    \end{align}
    where $z^{\text{ALG}}(I)$ is the makespan of ALG, $z^*(I)$ is the optimal makespan with perfect anticipation, and $S$ are the picking locations of instance $I$. 
\end{lemma}
\begin{proof}
The Lemma is proven by induction in Appendix~\ref{sec:proof_limited_relocation_gain}
 \end{proof}
\section{Results}\label{sec:experiments}

E-commerce warehouses strive to balance low costs with quick order turnovers, though these objectives often conflict. 
To sharpen our observations, we analyze two idealized systems focusing  solely on one objective, either  \textit{makespan} minimization (representing picker's time and wage costs) or order \textit{turnover} minimization (representing average order turnover). Although these systems do not exist in pure form, our main conclusions align and lead to actionable recommendations that improve \textit{both objectives simultaneously}.

For our experiments, we simulate a two-block rectangular warehouse with $a=3$ cross-aisles and $b=10$ aisles, following standard simulation frameworks from the literature. The depot is located on the central cross-aisle on the left side of the warehouse. The picker moves at $v=$0.8~m/sec and spends $t^p=$10~sec to collect an item at each picking location. In the \textit{Basis} setting orders average 1.5 items (\textit{Smallorders}), the order arrival rate is $r=200$ orders/8h and the picking cart has a capacity of $c=2$ bins. For comparison, orders at German Amazon warehouses average 1.6 items \citep{Boysen2019}. 
We also generate three additional settings using a one-factor-at-a-time design by examining larger orders (\textit{Largeorders}, with 2.5 items per order on average), an increased order arrival rate ($r=250$ orders/8h), and a larger cart capacity ($c=4$ bins). We use the designed algorithm DP, which can solve instances with more than 45 item locations in a reasonable time, to generate CIOSs. We  independently generate 20 instances per setting with $n^{o}=15$ orders each, a standard instance size for similar types of analysis in the warehousing literature \citep[cf.][]{Bukchin2012}. Additionally for the makespan objective, we confirm the \textit{representativeness} of our analysis for larger instances using the nonparametric \textit{Kruskal-Wallis test (KWT)} on independently generated \textit{Basis} instances with $n^{o}=15, 18, $ and $n^{o}=21$ orders. Under general assumptions, the KWT assesses whether the measured metrics for instances with $n^{o}=15$ orders remain consistent for instances with $n^{o}=18 $ and $n^{o}=21$ orders. We expect the outcome of KWT for the turnover objective to mirror that of the makespan. 
In total, we use $20\cdot(4+2)=120$ instances in the analysis of CIOSs. 

For the \textit{makespan} objective, we computed the CIOSs for all instances using the \textit{exact} DP approach from Section~\ref{sec:cost_labeling}. After validating its performance in Section~\ref{sec:exp_performance_DP}, we generated the associated CIOSs for the \textit{turnover} objective with \textit{heuristic} DP from Section~\ref{sec:turnover_heuristic}.   

All experiments have been conducted on the Compute Canada cluster using a maximum of 350 GB RAM and only one CPU for the DP approaches, 
but allowed the parallel computation with 32 CPUs (1 GB RAM each) per instance for the MIP approaches used as a benchmark in Section~\ref{sec:CIOS_analysis}. 
The DP approaches were implemented with Python 3.11.4. and  Gurobi 11.0.0. was used to solve the MIP models.


In the following, after validating the designed DP algorithms in Section~\ref{sec:exp_performance_DP}, the analysis of CIOSs is presented in Section~\ref{sec:CIOS_analysis}. Based on this analysis, Section~\ref{sec:good_online_algs} designs and discusses good myopic online policies and Section~\ref{sec:anticipationQ} estimates the remaining potential of anticipation.

\subsection{Performance of the designed DP algorithms} ~\label{sec:exp_performance_DP}
Tables~\ref{tab:performance_DP} and ~\ref{tab:performance_DP_turnover} validate the  \textit{exact} DP approach from Section~\ref{sec:cost_labeling} and the \textit{heuristic} DP approach from Section~\ref{sec:turnover_heuristic} by comparing their performance to a standard MIP solver (refer to Appendix ~\ref{sec:MIPs} for MIP models), using a one-hour runtime limit. The MIP solver had the advantage of utilizing parallelization with 32 CPUs, while the DP approaches did not use parallelization.

\begin{table}
\centering
\caption{Performance of exact DP-approach for \textbf{makespan} objective compared to benchmarks for runtime limit of 3600 sec}   \label{tab:performance_DP}
\scriptsize{
\begin{tabular}{lrr|rrr|rrr|rrr}
\toprule
\multicolumn{3}{c|}{Instances} &\multicolumn{3}{c|}{DP-approach} & \multicolumn{3}{c|}{Gains of DP-dominance rules} & \multicolumn{3}{c}{MIP solver}  \\ [0.05cm]
 & & items $n^i$  & \multicolumn{2}{c}{Runtime (sec)} & & & \multicolumn{2}{c|}{Runtime reduction} & & \multicolumn{2}{r}{Opt. gap of UB} \\
Setting & $n^o$  & med (max)    & \quad avg  & max  & \# opt & \#opt & \qquad avg & max & LB & \qquad  \qquad avg & max\\
\midrule
$LargeOrders\_c2\_r200$ & 15 & 37.5 (44) & 1478 & 3430 & 18 & 0 & 26\% & 85\% & -- & 11.6\% & 37.3\% \\
$SmallOrders\_c2\_r200$ & 15 & 22 (26) & 51 & 119 & 20 & 0 & 49\% & 99\% & -- & 0.5\% & 3.7\% \\
$SmallOrders\_c2\_r250$& 15 & 22.5 (27) & 74 & 123 & 20 & 0 & 27\% & 82\% & -- & 0.9\% & 4.1\% \\
$SmallOrders\_c4\_r250$ & 15 & 22.5 (25) & 1428 & 3226 & 20 & 3 & 30\% & 72\% & -- & 0.8\% & 4.4\% \\
$SmallOrders\_c2\_r200$ & 18 & 27 (30) & 755 & 1531 & 20 & 0 & 43\% & 96\% & -- & 3.1\% & 13.0\% \\
$SmallOrders\_c2\_r200$ & 21 & 29 (34) & 1216 & 3396 & 5 & 5 & -- & -- & -- & 0.0\% & 0.1\% \\
\midrule
\multicolumn{12}{l}{Abbreviations. med: median; avg:average; max:maximum; \# opt: number instances out of 20 solved to optimality; LB: lower bound}\\
\multicolumn{12}{l}{ \phantom{Abbreviations}  found by MIP-solver; Opt.gap of UB: gap between upper bound found by MIP-solver and optimum found by DP.}\\
\multicolumn{12}{l}{ Note. All values refer to the subset of instances solved by DP within 3600 sec }\\
  \end{tabular}}
\end{table}

\begin{table}
\centering
\caption{Performance of heuristic DP-approach for \textbf{turnover} objective compared to MIP solver for runtime limit of 3600 sec}   \label{tab:performance_DP_turnover}
\scriptsize{
\begin{tabular}{lrr|rrrr|rrrr}
\toprule
\multicolumn{3}{c|}{Instances} &\multicolumn{4}{c|}{DP-approach}  & \multicolumn{4}{c}{MIP solver}   \\ [0.05cm]
 & & items $n^i$  & \multicolumn{2}{c}{Runtime (sec)} & & &   & \multicolumn{3}{r}{gap of UB to DP sol} \\
Setting & $n^o$  & med (max)    & \quad avg  & max  & \# sol & \# sol $\leq$ MIP UB &  LB & \qquad  \qquad avg & max & min \\
\midrule
$LargeOrders\_c2\_r200$ & 15 & 37 (40) & 1498 & 2512 & 16 & 16 (100\%) &  -- & 38\% & 141\%  & 7\%\\
$SmallOrders\_c2\_r200$ & 15 & 22 (26) & 62 & 138 & 20 & 18 (90\%) & -- & 9\% & 49 \% & -7\% \\
$SmallOrders\_c2\_r250$& 15 & 22.5 (27) & 93 & 159 & 20 & 20 (100\%)  & -- & 8\% & 26\% & 1\%\\
$SmallOrders\_c4\_r250$ & 15 & 22 (25) & 1576 & 3267 & 19 & 19 (100\%) & -- & 27\% & 169\% & 2\% \\
$SmallOrders\_c2\_r200$ & 18 & 27 (30) & 955 & 1801 & 20 & 20 (100\%)&  -- & 13\% & 32\% &1 \\
$SmallOrders\_c2\_r200$ & 21 & 30 (34) & 810 & 2410 & 4 & 4 (100\%) & --  & 18\% & 27\% & 5\%\\
\midrule
\multicolumn{11}{l}{Abbreviations. med: median; avg:average; max:maximum; \# sol: number instances out of 20 solved heuristically; LB: lower bound}\\
\multicolumn{11}{l}{ \phantom{Abbreviations}  found by MIP-solver; gap of UB to DP sol: gap between upper bound found by MIP-solver and DP solution.}\\
\multicolumn{11}{l}{ Note. All values refer to the subset of instances solved by DP within 3600 sec }\\
  \end{tabular}}
\end{table}

The MIP solver was not able to solve \textit{any} instance to proven optimality, likely due to weak lower bounds. 

For the \textit{makespan} objective,  DP solved almost all (103 out of 120) instances to \textit{proven optimality} (see Table~\ref{tab:performance_DP}). Customized dominance rules were highly effective, reducing the DP's runtime by an average of 36\%, and up to 99\% in some cases, enabling the optimal solution of 8 additional instances.

For the \textit{turnover} objective (see Table~\ref{tab:performance_DP_turnover}), the DP produced significantly better solutions than the MIP solver, with 98\% of the solutions improving upon the MIP's upper bound.

For our analysis of the CIOSs in the following sections, we ran the DP without a time limit for both objectives. Specifically, we computed \textit{optimal solutions for all 120 instances} in the dataset for the makespan objective. 
Overall, Tables~\ref{tab:performance_DP} and~\ref{tab:performance_DP_turnover} emphasize that without the customized design of the DP algorithms, the CIOS analysis in the following sections would have been impossible. 

\subsection{Analysis of complete-information optimal solutions (CIOSs)} \label{sec:CIOS_analysis}

Depending on the objective, the operations of the warehouse change. 
The \textit{makespan} is effectively minimized in warehouses with high order arrival rates: With more frequent arrivals, long queues allow the picker to select orders whose items are located close to each other into the same batch, thus, reducing the picking time per order. High arrival rates also decrease idle time, as there are fewer moments when no orders are available. As a result, most batches in CIOSs utilize the full capacity of the picking cart (see Figure~\ref{fig:Batchsize}). For a four-bin cart (\textit{Small\_c4\_r200}), an average of 53\% of all batches are composed of four orders.  In the remaining settings with two-bin carts, batches of two orders account for 77-86\% of all batches.

In contrast, warehouses that prioritize average \textit{turnover} achieve the lowest turnover times at low order arrival rates. Here, the picker often remains idle at the depot, enabling immediate picking upon an order's arrival, usually in a single-order batch. Consequently, 48\% of CIOS batches are single-order batches (see Figure~\ref{fig:Batchsize}).  Indeed, picking a smaller order as a single-order batch first can offset moderate savings from collecting multiple orders simultaneously, resulting in a smaller average order turnover.

\begin{figure}
    \begin{subfigure}[b]{0.215\textwidth}
        \includegraphics[width=\textwidth]{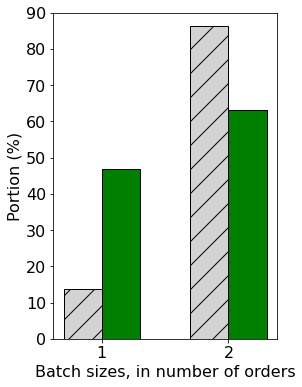}
        \caption*{$LargeOrders\_c2\_r200$}
        \label{fig:a}
    \end{subfigure}
    ~ 
    \begin{subfigure}[b]{0.2\textwidth}
        \includegraphics[width=\textwidth]{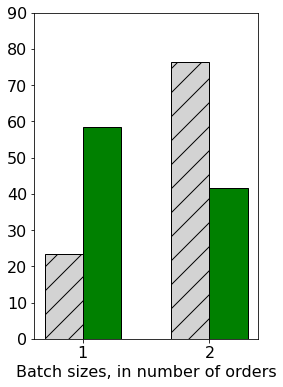}
        \caption*{$SmallOrders\_c2\_r200$}
        \label{fig:b}
    \end{subfigure}
    ~
    \begin{subfigure}[b]{0.2\textwidth}
        \includegraphics[width=\textwidth]{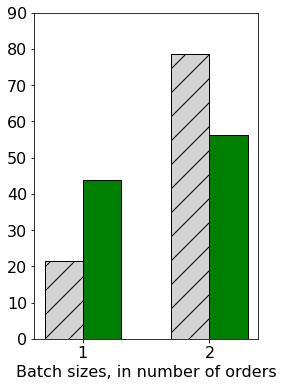}
        \caption*{$SmallOrders\_c2\_r250$}
        \label{fig:b}    
    \end{subfigure}
    ~ 
    \begin{subfigure}[b]{0.28\textwidth}
        \includegraphics[width=\textwidth]{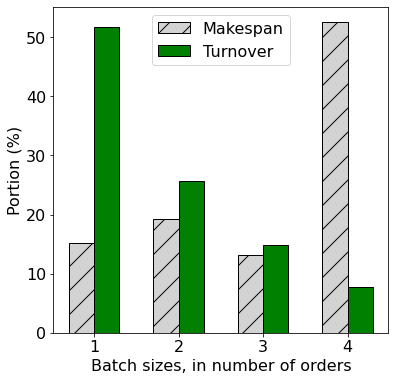}
        \caption*{$SmallOrders\_c4\_r200$}
        \label{fig:b}
    \end{subfigure}
\caption{Portion of batches of different sizes}\label{fig:Batchsize}    
\end{figure}

We continue by examining batching (Section~\ref{sec:CIOS_batching}), routing (Section~\ref{sec:CIOS_routing}), and anticipation in CIOSs. Section~\ref{sec:CIOS_strategic_waiting} covers strategic waiting, with additional forms of anticipation discussed in Section~\ref{sec:CIOS_anticipation}. We focus on identifying  decisions that improve \textit{both} objectives simultaneously. For clarity, we compare certain decisions in CIOSs with two common planning heuristics: \textit{first-in-first-out (FIFO)} batching and \textit{myopic reoptimization (Reopt)}. FIFO-batching forms batches by assigning orders as they arrive, while Reopt is a myopic policy that reoptimizes based on the current queue of available orders each time a new order arrives.

\subsubsection{Batching in CIOSs} \label{sec:CIOS_batching}
Batching decisions are the most costly in terms of runtime when generating CIOS. 

\begin{table}
\caption{Batching decisions in CIOSs compared to those in FIFO }\label{tab:ordersequencing}
\centering
    \scriptsize{
\begin{tabular}{crc|cc|cc}
\toprule
&&&\multicolumn{2}{c|}{\underline{Makespan objective} \vspace{0.5mm}}& \multicolumn{2}{c}{\underline{Turnover objective}}\\
& & & \# batches with & \# orders with & \# batches with & \# orders with\\
 & Setting & $n^o$ & non-consecutive orders & altered positions& non-consecutive orders & altered positions\\
\midrule
Analysis & $LargeOrders\_c2\_r200$  & 15 & 2.7 (34\%) & 5.1 (34\%) & 2.6 (29\%) & 6.5 (43\%)\\
 & $SmallOrders\_c2\_r200$  & 15 & 2.2 (26\%) & 4.0 (26\%) & 1.6 (16\%) & 5.0 (33\%)   \\
 & $SmallOrders\_c2\_r250$   & 15 & 1.9 (23\%) & 4.2 (29\%)  & 2.3 (24\%) & 5.7 (38\%) \\ 
& $SmallOrders\_c4\_r250$  & 15 & 1.3 (25\%)  & 4.0 (26\%) & 1.0 (14\%)  & 5.4 (36\%) \\
& Overall & 15 & 2.0 (27\%) & 4.3 (29\%)  & 1.8 (21\%) & 5.6 (37\%) \\
\midrule
Validation & $SmallOrders\_c2\_r200$  & 15 & 26\% & 26\% &&\\
 & $SmallOrders\_c2\_r200$  & 18 & 23\% & 28\% &&\\
 & $SmallOrders\_c2\_r200$   & 21 & 18\% & 17\% &&\\
&  \multicolumn{2}{r|}{Kruskal Wallis test p-values} & 0.35 & 0.08 &&\\
\midrule
\multicolumn{7}{l}{Note. Average number of batches with non-consecutive orders and average number of orders that deviate from their FIFO position, per instance. }\\
\end{tabular}
}
\end{table}

Table~\ref{tab:ordersequencing} shows significant differences between CIOS and FIFO batches for both examined objectives -- \textit{makespan} and \textit{turnover}. Unlike FIFO, 23\%-34\% of CIOS batches do not contain consecutively arrived orders for the former objective, 14\%-29\% of batches for the latter. 
When arranging orders in each CIOS batch by  arrival time for the \textit{makespan} objective, about 30\% of orders are at altered positions in the resulting sequence compared to FIFO.  The order sequence appears more important for  warehouses operating under the \textit{turnover} objective (overall share of orders with altered positions of 37\%).

A primary reason for deviating from the FIFO order sequence in batching is to improve the matching of jointly collected orders. \textit{In case of $c=2$}, this matching quality can be measured as savings \citep[cf.][]{deKoster1999}. The \textit{savings} $\psi(B)$ for a batch $B$ is defined as the difference between the shortest picking time for $B$ (noted $\chi(B)$) and the sum of the individual shortest picking times for each order in $B$ (noted $\chi(o), o \in B$). The measure $\psi(B)$ is normalized from 0 to 1, where 0 indicates no savings (picking each order separately takes the same time as picking them jointly) and 1 indicates the maximum possible savings (picking both  orders separately takes twice as long as picking them jointly):
\begin{align}
\psi(B) = \frac{\sum\limits_{o \in B} \chi(o)-\chi(B)}{\chi(B)}
 \end{align}
 For settings with $c=2$, Table~\ref{tab:batch_savings} compares the savings in two-order CIOS batches with those of \textit{randomly} selected two-order batches. For the random batches, we sample from all possible two-order combinations for the same instance. Observe that CIOS batches have higher savings: about 70\% of CIOS batches exceed the median savings of randomly generated batches for the \textit{makespan} objective, and 77\% for the \textit{turnover} objective, respectively. The greater savings for the \textit{turnover} objective can be attributed to higher opportunity costs associated with batching in this case, as discussed in the introduction to Section~\ref{sec:CIOS_analysis}.

\begin{table}
\centering
\caption{Quality of batches in CIOS: Savings of two-order batches in settings with batching capacity $c=2$}\label{tab:batch_savings}
    \scriptsize{
    \begin{tabular}{ccc|cc|cc|c}
    \toprule
     & & & \multicolumn{2}{c|}{\underline{Makespan objective} \vspace{0.5mm}}& \multicolumn{2}{c|}{\underline{Turnover objective}} & \\
    & & & Mean CIOS  & \# CIOS batch savings  & Mean CIOS & \# CIOS batch savings & Mean random\\
      & Setting & $n^o$ &  batch savings & $>$ random median& batch savings & $>$ random median& batch savings\\
     \midrule
  Analysis & $LargeOrders\_c2\_r200$  & 15 & 0.36 & 4.8 (69\%)& 0.37 & 4.2 (72\%) &0.28  \\ 
   & $SmallOrders\_c2\_r200$   & 15 & 0.33 & 4.4 (68\%) & 0.37 & 3.4 (79\%)& 0.26 \\ 
     & $SmallOrders\_c2\_r250$   & 15 & 0.37   & 4.7 (71\%)& 0.40 & 4.5 (81\%) & 0.27\\ 
     & Overall  & 15 & 0.35 & 4.6 (70\%) & 0.38 & 4.0 (77\%) & 0.27  \\ 
     \midrule
   Validation & $SmallOrders\_c2\_r200$   & 15 & 0.33 & 68\% &&& 0.26 \\  
   & $SmallOrders\_c2\_r200$   & 18 & 0.38  & 75\% &&& 0.27\\ 
     & $SmallOrders\_c2\_r200$  & 21 & 0.33 & 71\%&&& 0.26 \\ 
     & \multicolumn{2}{r|}{Kruskal Wallis test p-values}   & 0.02 & 0.18  &&& 0.77 \\ 
    \midrule
   \multicolumn{8}{l}{Note. \textit{Columns 4, 6,8}: Average of the mean savings of the formed 2-order batches per instance. \textit{Columns 5,7}: Number of 2-order batches }\\
    \multicolumn{8}{l}{ \phantom{Note} whose savings are larger than the median of the random 2-order batches of this instance, averaged over all instances. }
    \end{tabular}
    }
\end{table}

\begin{observation}[Batching]\label{obs:batching1}
Batching in CIOSs differs significantly from the myopic FIFO-rule, in particular, CIOS batches show a deliberate improvement in the quality of order matching.    
\end{observation}

Upon a closer examination, batching in CIOS differs from most batching policies discussed in the warehousing literature (see Section~\ref{sec:lit_review} for an overview) in one key aspect:  most orders are added to a batch \textit{after} the picker has left the depot to collect items of  this batch. This applies to 64\% and 62\% orders for the \textit{makespan} and \textit{turnover} objectives, respectively (see Table~\ref{tab:intervention}). Policies that allow modifications to an active batch, like in CIOS, are termed \textit{interventionist}, while \textit{non-interventionist} policies do not allow changes once a batch is started  \citep[cf.][]{Giannikas2017}. Similar to strategic waiting, intervention allows for savings from batching when order queues are short. However, since the picker does not stay idle, but continues moving towards known orders instead of waiting, intervention is a \textit{less costly} alternative to strategic waiting in terms of the opportunity costs, especially if no new orders arrive quickly.

\begin{observation}[Intervention in Batching]\label{obs:Intervention}
Most orders in CIOSs are added to a batch dynamically, \textit{after} the picker's departure from the depot to collect the items of  this batch.
\end{observation}

\begin{figure}
\begin{minipage}{0.65\textwidth}
\centering
\captionof{table}{Orders added by intervention in CIOS}\label{tab:intervention}
\scriptsize{
\begin{tabular}{@{}crc|c|c}
\toprule
& & &\underline{Makespan objective} \vspace{0.5mm}& \underline{Turnover objective}\\
 & Setting & $n^o$ & \# orders & \# orders\\
\midrule
Analysis & $LargeOrders\_c2\_r200$ & 15 & 7.7 (5\%) & 7.8 (52\%) \\
& $SmallOrders\_c2\_r200$  & 15 & 10.3 (69\%) & 10.7 (71\%) \\
 & $SmallOrders\_c2\_r250$ & 15  & 8.8 (59\%) & 8.9 (59\%)\\ 
& $SmallOrders\_c4\_r250$ & 15 & 11.4 (76\%)& 10.1 (67\%)\\
& Overall & 15 & 9.5 (64\%) & 9.4 (62\%)\\
    \midrule
Validation &  $Smallorders\_c2\_r200$  & 15 &  69\%& \\ 
 & $Smallorders\_c2\_r200$  & 18 &  62\% &\\
 & $Smallorders\_c2\_r200$  & 21 &  69\%& \\
 & \multicolumn{2}{r|}{Kruskal-Wallis test; p-value} & 0.58 &\\
    \midrule    
    \multicolumn{5}{l}{Note. Average number of orders per instance added by intervention.}
\end{tabular}}
\hfill
\end{minipage}
\begin{minipage}{0.33 \textwidth}
\includegraphics[scale=0.5]{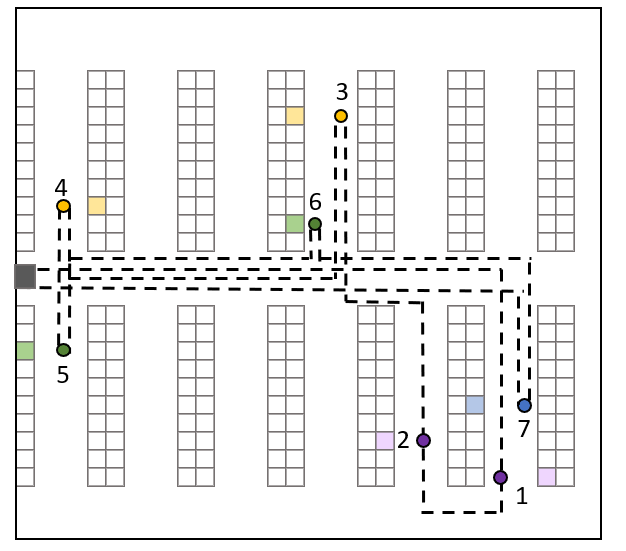}
\captionof{figure}{Exemplary CIOS batch route} \label{fig:strage_route}
\scriptsize{The numbers indicate the visiting sequence. A CIOS route for an instance in $SmallOrders\_c4\_r250$ with $n^o=15$ and the \textit{makespan} objective }
\label{fig:strategic_relocation}  
\end{minipage}
\end{figure}

\subsubsection{Routing in CIOSs} \label{sec:CIOS_routing}

Figure~\ref{fig:strage_route} shows a typical picker route for completing a batch in CIOSs, featuring three horizontal direction changes while crossing the central cross-aisle and  two cases where sub-aisles (4th upper and 6th lower) are entered  twice from the same side. We observed CIOS batch routes with up to five horizontal direction changes within the cross-aisles.
This contrasts sharply with common routing heuristics like \textit{S-shape} or \textit{largest-gap} policies, where sub-aisles are entered at most once from each side, and  horizontal direction changes while traversing the cross-aisles are limited to one and three, respectively \citep[see][for a detailed discussion]{Roodbergen2001a,LeDuc2007}. The key factor driving the CIOS routing pattern is \textit{intervention} (cf. Observation~\ref{obs:Intervention}), where newly arrived orders are added to an active batch. 

\begin{observation}[Routing] \label{obs:Routing} CIOS routing differs significantly from common routing policies.
\end{observation}

\subsubsection{Strategic waiting} \label{sec:CIOS_strategic_waiting}

As discussed in Section~\ref{sec:intro}, strategic waiting -- deliberately delaying picking to enhance future operations -- is one of the least understood concepts in the warehousing literature. 

Based on the analysis of CIOSs, we identify two types of strategic waiting. Note that, since CIOS operates with perfect anticipation, such waiting occurs at the first picking location of the future order:
\begin{itemize}
\item \textit{Waiting for batch extension (B.ext.)}: The picker has collected all items for current orders but, instead of delivering the collected orders to the depot, waits in the field for the next order to add to the active batch, increasing savings from batching and avoiding additional trips to the depot. 
\item \textit{Waiting for a better fit (B.fit)}: The picker waits at the picking location of a future order, even though items of other orders are available for picking, because the future order better fits  the current batch route.
\end{itemize}

\begin{table}
    \centering
    \caption{Strategic waiting in CIOSs}\label{tab:waiting_types_CIOS}
    \scriptsize{
    \begin{tabular}{@{}ccc|c|ccc|c@{}c@{}c|c@{}c@{}c@{}}
    \toprule
    & & & & \multicolumn{3}{c}{Avg. total waiting time} &  \multicolumn{3}{c}{Waiting for batch extension} & \multicolumn{3}{c}{ Waiting for a better fit} \\
    Obj.  & Setting &  $n^o$ & Makesp.  & Idleness & B.ext  & B.fit & \# orders &  \ Avg time & \ 80P time &\# orders & \ Avg time & \ 80P time \\
     \midrule
 \multirow{5}{*}{\rotatebox[origin=c]{90}{\makecell{Makesp. \\ Analysis }}} & $LargeOrders\_c2\_r200$ & 15 & 2488  & 277 (10\%) & 15 (1\%) & 16 (1\%) & 0.6 (4\%) & 27 &55  & 0.8 (5\%) & 22 & 32 \\ 
  & $SmallOrders\_c2\_r200$ & 15  & 2447 &  729 (27\%) & 77 (3\%) & 37 (1\%) & 1.1 (7\%) & 70 & 126 & 1.1 (7\%) & 35 & 46  \\ 
   & $SmallOrders\_c2\_r250$ & 15  & 1975  & 320 (15\%) & 58 (3\%) & 23 (1\%)  & 1.1 (7\%) & 55 & 101 & 0.7 (5\%) & 35 & 50 \\ 
    & $SmallOrders\_c4\_r250$ & 15  & 1773 & 193 (9\%) & 111 (6\%) & 83 (4\%) &  1.7 (11\%) & 66  & 105 & 2.3 (15\%) & 36 & 67  \\
    & Overall & 15 & 2171 & 380 (15\%) & 65 (3\%) & 40 (2\%) & 1.1 (7\%) & 59 & 106 & 1.2 (8\%) & 33 & 48\\
    \midrule
   \multirow{4}{*}{\rotatebox[origin=c]{90}{\makecell{Makesp.\\ Valid.}}}  & $SmallOrders\_c2\_r200$  & 15 & 2447 &  27\% & 3\% & 1\% & 7\%  & 70 & 126 & 7\% & 35 & 46   \\
   & $SmallOrders\_c2\_r200$  & 18 & 2571 &  15\% & 7\% & 2\% &  12\%  &   92 & 166 & 6\%& 39 & 36   \\
   & $SmallOrders\_c2\_r200$  & 21 & 3395 & 23\% & 6\% & 2\% &  14\%  & 73  & 123 & 5\% & 52 & 69   \\  
   & \multicolumn{2}{r|}{Kruskal Wallis test p-values} & -- &  0.01 & 0.02 & 1.00 & < 0.01  & 0.43 & --& 0.97 &  0.57 & -- \\ 
    \midrule
   \multirow{4}{*}{\rotatebox[origin=c]{90}{\makecell{Turnover}}} & $LargeOrders\_c2\_r200$ & 15 & 2569 & 275 (10\%) & 3 (0\%) & 17 (1\%) &  0.1 (1\%)  & 34 & 51 & 0.8 (5\%) & 22 & 37   \\
  & $SmallOrders\_c2\_r200$ & 15  & 2496 & 791 (29\%) & 2 (0\%) & 16 (1\%) &  0.2 (1\%) & 12 & 20 & 1.2 (8\%) & 14 & 21   \\
     &$SmallOrders\_c2\_r250$ & 15  & 2016  &  374 (17\%) & 6 (0\%) & 11 (1\%) & 0.5 (3\%)  & 14 & 21 & 0.8 (5\%) & 14 & 19 \\ 
   & $SmallOrders\_c4\_r250$ & 15 & 1842 &  294 (14\%) & 10 (0\%) & 13 (1\%) &  0.6 (4\%) & 17 & 32 & 1.0 (7\%) & 13 & 22  \\ 
   & Overall & 15 & 2225 & 434 (17\%) & 5 (0\%) & 14 (1\%) & 0.3 (2\%) & 17 & 32 & 0.9 (6\%) & 15 & 24 \\
    \midrule
   \multicolumn{13}{l}{Note.\textit{Columns 4-6}: Average total time per instance  spent waiting (as \% of the makespan).}\\
   \multicolumn{13}{l}{ \phantom{Note.}\textit{Columns 7 and 10, \# orders}: Average number of orders per instance waited for in type B.ext- and B.fit-waiting, respectively. \textit{Columns 8} } \\
   \multicolumn{13}{l}{ \phantom{Note.}\textit{and 11 (9 and 12), Avg time (80P time)}: Average (80-percentile) duration of B.ext- and B.fit-waiting, respectively.} \\
    \end{tabular}
    }
\end{table}

Table~\ref{tab:waiting_types_CIOS} compares strategic waiting and \textit{idle time} in CIOSs. The latter occurs when all available orders have been picked and delivered to the depot. For both the \textit{makespan} and  \textit{turnover} objectives, strategic waiting is significantly less than  idle time. 
When measured in terms of the makespan, strategic waiting accounts for merely 5\% of the makespan on average, compared to  15\% of idle time for the \textit{makespan} objective. For the \textit{turnover} objective, it is even 1\% of waiting compared to 17\% of idle time. Notably, while strategic waiting is frequent -- occurring before 6\% to 26\% of orders  -- it is very brief, lasting just 12-70 seconds on average. 

This explains the monotonous decrease in performance for both objectives of the Reopt policy when adopting a \textit{variable time window strategy} (VTW) \citep[cf.][]{GilBorras2024} which enforces picker waiting times at the depot, until there are at least $N$ orders in the queue, see Figure~\ref{fig:waiting_lineplots}. Even 
at $N=2$, the waiting affects too many orders and is significantly longer than CIOS waiting (more than 92-184 seconds on average), see Table~\ref{tab:waiting_times_VTW}. Furthermore, 21-52\% of the policy's waiting was \textit{misplaced}, as the associated waiting times $w_j$ for the second order's arrival exceeded the time $\chi(o_j)$ in which the delayed order $o_j$ could have been completed in a single-order batch. This revises recommendations on optimal waiting times based on the \textit{variable time window batching policy (VTWB)} \citep{Vannieuwenhuyse2009}, where the gain from waiting is a \textit{concave} function in $N$, with maximum gain possibly achieved when waiting for a minimum queue of $N\geq2$. Our analysis shows the gain stems from the restrictive FIFO-batching and the lack of intervention in VTWB and is nullified when applying optimal batching and intervention (as in Reopt).   

\begin{figure}
    \begin{subfigure}[b]{0.4\textwidth}
        \includegraphics[width=\textwidth]{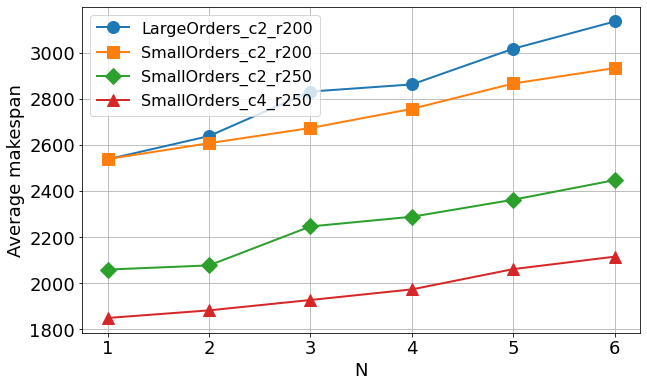}
        \caption*{(a) Makespan objective}
        \label{fig:pushcart}
    \end{subfigure}
    ~ \hfill
    \begin{subfigure}[b]{0.4\textwidth}
        \includegraphics[width=\textwidth]{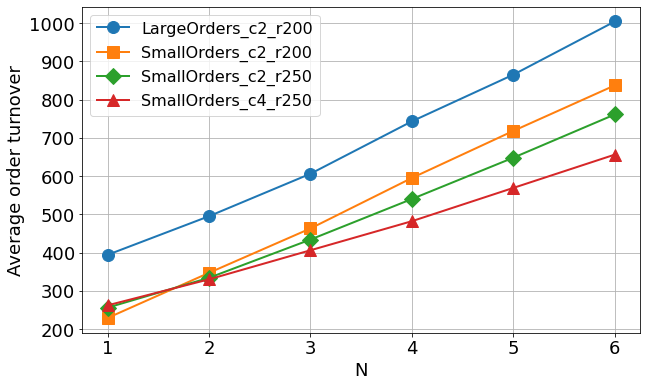}
        \caption*{(b) Turnover objective}
        \label{fig:pushcart}
    \end{subfigure}
\caption{Performance of Reopt when waiting for $N$ orders at the depot (variable time window strategy)}\label{fig:waiting_lineplots}    
\end{figure}

\begin{table}
\caption{Waiting times in solutions of Reopt under variable time window strategy with $N=2$}\label{tab:waiting_times_VTW}
\centering
\scriptsize{
\begin{tabular}{l|rrr|rrr}
\toprule
& \multicolumn{3}{c}{\underline{Makespan objective} \vspace{0.5mm}} & \multicolumn{3}{c}{\underline{Turnover objective}}  \\
Setting ($n^o=15$) & \# orders & Avg time & \# $w_j>\chi(o_j)$ & \# orders & \qquad Avg time & \# $w_j>\chi(o_j)$\\
\midrule
$LargeOrders\_c2\_r200$ \qquad \qquad  & 2.3 (15\%) & 134 & 22\% & 3.3 (22\%) & 116 & 23\%   \\ 
$SmallOrders\_c2\_r200$  & 3.8 (25\%) & 164 & 43\% & 4.5 (30\%) & 184 & 52\%   \\ 
$SmallOrders\_c2\_r250$ & 2.9 (19\%) & 103 & 21\% & 3.7 (25\%) & 115 & 33\% \\ 
$SmallOrders\_c4\_r250$  & 2.8 (19\%) & 92 & 21\% & 3.9 (26\%) & 95 & 27\%  \\ 
Overall & 3.0 (20\%) & 123 & 27\% & 3.9 (26\%) & 128 & 34\% \\
    \midrule
    \multicolumn{7}{l}{Note. \textit{Columns 2 \& 5, \#orders:} Average number of orders per instance delayed because of waiting for the second order. }\\
    \multicolumn{7}{l}{\phantom{Note.} \textit{Columns 3 \& 6, Avg time:} Average duration of waiting, after first order has arrived. \textit{Columns 4 \& 7, $w_j>\chi(o_j)$:}}\\
    \multicolumn{7}{l}{\phantom{Note.}  Percentage of delayed orders that could have been picked in a one-order batch within waiting time. }\\
\end{tabular}
}
\end{table}

We also assessed the common advice in the literature to avoid waiting if there are enough orders in the queue to complete a batch \citep[cf.][]{Henn2012, Alipour2018, GilBorras2024, Giannikas2017}. Contrary to this advice, in 28\% and 41\% of the observed cases of waiting for a better fit in CIOSs, for the \textit{makespan} and \textit{turnover} objective, respectively, enough orders were available to form a full batch.

Differences in strategic waiting behavior in CIOSs are observed depending on the warehouse's objective. 

Since one-order batches dominate under the \textit{turnover} objective, strategic waiting aimed at better batching is less relevant.  This is especially true for waiting for batch extensions, which is negligible—meaning the picker almost always returns orders to the depot when the order queue is empty.  On average, strategic waiting accounts for only about 1\% of the makespan and lasts just 12-17 seconds per occurrence in settings with small orders. For larger orders, where even one-order batches take longer, the tolerance for waiting is slightly higher (22-34 seconds per occurrence), though strategic waiting happens less frequently. 

\begin{observation}[Turnover: Strategic Waiting] \label{obs:Waiting}
Under the \textit{turnover} objective, only a negligible amount of time is spent on strategic waiting, and each occurrence of strategic waiting is brief and well-timed.  
\end{observation}

For the \textit{makespan} objective, waiting for batch extension becomes more important than waiting for a better fit, accounting for 3\% and 2\% of the makespan, respectively. Additionally, when cart capacity is large ($c=4$), creating more opportunities for batching, strategic waiting increases further, reaching 6\% and 4\% of the makespan for both types of strategic waiting, respectively. In this case, idle time is reduced ("cannibalized") by earlier  strategic waiting (9\% idle time for $c=4$ compared to the 15\% overall average).

An interesting aspect of the \textit{makespan} objective is that the precise \textit{placement} of strategic waiting is almost irrelevant, as explained in Lemma~\ref{lemma:waitingtime_placement} (Section~\ref{sec:analytics_makespan}). 
Table~\ref{tab:waiting_at_beginning} tests the following hypothesis: Unlike the \textit{turnover} objective, where both the amount and timing of waiting are critical, for the \textit{makespan} objective, it may be sufficient to anticipate only the \textit{amount of waiting}.   By positioning this at the start and gaining additional information on incoming orders, we could potentially improve operations without needing to predict future orders in detail.  To test this, we conducted an analysis where, for each instance $I$, we enforced initial waiting of $w^*(I)$ and applied myopic reoptimization (Reopt) afterwards. We calculated $w^*(I)$ as the sum of the total idle time and the strategic waiting time in the CIOS for this instance. As shown in Table~\ref{tab:waiting_at_beginning}, the results of Reopt did not improve in most cases, except for two instances where the makespan improved by 2.4\% and 6.6\%, respectively. This indicates that the benefits of additional information from enforced initial waiting were outweighed by the cost of waiting. In other words, strategic waiting alone is insufficient and must be paired with adjustments in order batching and sequencing based on the anticipation of future orders.

\begin{table}
\caption{The impact of prior waiting for $w^*(I)$ on Reopt performance for the \textit{makespan} objective}\label{tab:waiting_at_beginning}
\centering
    \scriptsize{
\begin{tabular}{r|cc}
\toprule
  Setting ($n^o=15$) & Average improvement & Best improvement\\
\midrule
 $LargeOrders\_c2\_r200$   & 0.0\% & 0.0\% \\
  $SmallOrders\_c2\_r200$   & 0.1\% & 2.4\%  \\
  $SmallOrders\_c2\_r250$    & 0.0\% & 0.2\%  \\ 
 $SmallOrders\_c4\_r250$  & 0.3\%  & 6.6\% \\
 Overall & 0.1\% & 6.6\%  \\
\midrule
\end{tabular}
}
\end{table}


\begin{observation}[Makespan: Strategic Waiting] \label{obs:Waiting}
For \textit{makespan} objective, strategic waiting becomes more important with larger carts due to increased batching opportunities. Only the total waiting time matters, as it can be positioned at the start. However, to leverage strategic waiting, future order anticipation and adjustments in order batching and sequencing are necessary. 
\end{observation}

\subsubsection{Further forms of anticipation in CIOSs} \label{sec:CIOS_anticipation}
In addition to strategic waiting, CIOSs employ other forms of anticipation. The first involves anticipatory adjustments in order batching, while the second involves anticipatory adjustments in routing.

\begin{table}
\centering
\caption{Quality of batches in Reopt: Savings of two-order batches in settings with batching capacity $c=2$}\label{tab:batch_savings_Reopt}
    \scriptsize{
    \begin{tabular}{cc|cc|cc}
    \toprule
      & & \multicolumn{2}{c|}{\underline{Makespan objective} \vspace{0.5mm}}& \multicolumn{2}{c}{\underline{Turnover objective}}  \\
     & & Mean Reopt  & \# Reopt batch savings  & Mean Reopt & \# Reopt batch savings \\
       Setting & $n^o$ &  batch savings & $>$ random median& batch savings & $>$ random median\\
     \midrule
   $LargeOrders\_c2\_r200$  & 15 & 0.36 (-0.00) & 71\% (+2\%)& 0.39 (+0.02) & 79\% (+7\%)   \\ 
    $SmallOrders\_c2\_r200$   & 15 & 0.33 (-0.00) & 66\% (-2\%) & 0.40 (+0.03) & 92\% (+13\%) \\ 
      $SmallOrders\_c2\_r250$   & 15 & 0.36 (-0.01)  & 70\% (-1\%)& 0.46 (+0.06) & 87\% (+6\%) \\ 
      Overall  & 15 & 0.35 (-0.00)  & 69\% (-1\%) & 0.42 (+0.04) & 86\% (+9\%)  \\  
    \midrule
   \multicolumn{6}{l}{Note. \textit{Columns 3, 5}: Average of the mean savings of the formed 2-order batches per instance.  \textit{Columns 4,6}: Portion} \\
    \multicolumn{6}{l}{ \phantom{Note} of 2-order batches whose savings are larger than the median of the random 2-order batches of this instance, }\\
    \multicolumn{6}{l}{\phantom{Note} averaged over all instances. (In parenthesis, comparison to CIOS, see Table~\ref{tab:batch_savings})} 
    \end{tabular}
    }
\end{table}

To evaluate anticipatory batching, we compared the batching decisions in CIOSs with those in a purely myopic reoptimization policy (Reopt), by repeating the studies of Table~\ref{tab:batch_savings} for the savings of two-order batches in Reopt solutions, see Table~\ref{tab:batch_savings_Reopt}. Surprisingly, the savings of the Reopt and CIOS batches are almost equivalent, especially for the \textit{makespan} objective, where the overall mean batch savings have the equal value of 0.35, and where on average 69\% of the two-order batch savings exceeded the random median savings (compared to 70\% in CIOS). For the \textit{turnover} objective, the Reopt two-order batches 
even have relatively better savings compared to CIOS, which is also due to Reopt solutions forming \textit{fewer} two-order batches, given the lack of advanced knowledge in Reopt, but also the heuristic character of our DP for CIOS in the turnover objective, which sometimes overestimates the value of batching.

\begin{observation}[Anticipatory Batching] \label{obs:AnticipatoryBatching}
The batching quality measured in savings in Reopt solutions and CIOSs  is similar. 
\end{observation}

\begin{figure}
\begin{minipage}{0.65\textwidth}
\captionof{table}{Anticipatory routing -- strategic relocation in CIOS }\label{tab:strategicrelocation}
\centering
    \scriptsize{
\begin{tabular}{@{}crc|cc|cc@{}}
\toprule
& & & \multicolumn{2}{c}{\underline{Makespan objective}}  \vspace{0.5mm} & \multicolumn{2}{c}{\underline{Turnover objective}}\\
 & Setting & $n^o$ & Time & \# Orders &  Time & \# Orders \\
\midrule
Analysis & $LargeOrders\_c2\_r200$  & 15 & 213 (8\%) & 5.9 (39\%) & 215 (8\%) & 6.1 (41\%) \\
 & $SmallOrders\_c2\_r200$  & 15 & 381 (15\%) & 9.0 (60\%) & 386 (15\%) & 9.3 (62\%) \\
 & $SmallOrders\_c2\_r250$   & 15 & 269 (14\%) & 7.1 (47\%) & 274 (13\%) & 7.6 (50\%) \\ 
& $SmallOrders\_c4\_r250$  & 15 & 356 (19\%)  & 8.8 (58\%) & 314 (17\%) & 8.0 (53\%) \\
& Overall & 15 & 305 (14\%) & 7.7 (51\%) & 297 (13\%) & 7.7 (52\%) \\
\midrule
Validation & $SmallOrders\_c2\_r200$  & 15 & 15\% & 60\% & &\\
 & $SmallOrders\_c2\_r200$  & 18 & 15\% & 54\% & &\\
 & $SmallOrders\_c2\_r200$   & 21 & 16\% & 60\% & &\\ 
&  \multicolumn{2}{r|}{Kruskal Wallis test p-values} & 0.72 & 0.77 & & \\
\midrule
\multicolumn{7}{l}{Note. Average time picker spends with strategic relocation per instance, and average number}\\
\multicolumn{7}{l}{ \phantom{Notes.} of orders with strategic relocation before first pick}
\end{tabular}
}
\end{minipage}
\hfill
\begin{minipage}{0.3\textwidth}
\includegraphics[scale=0.35]{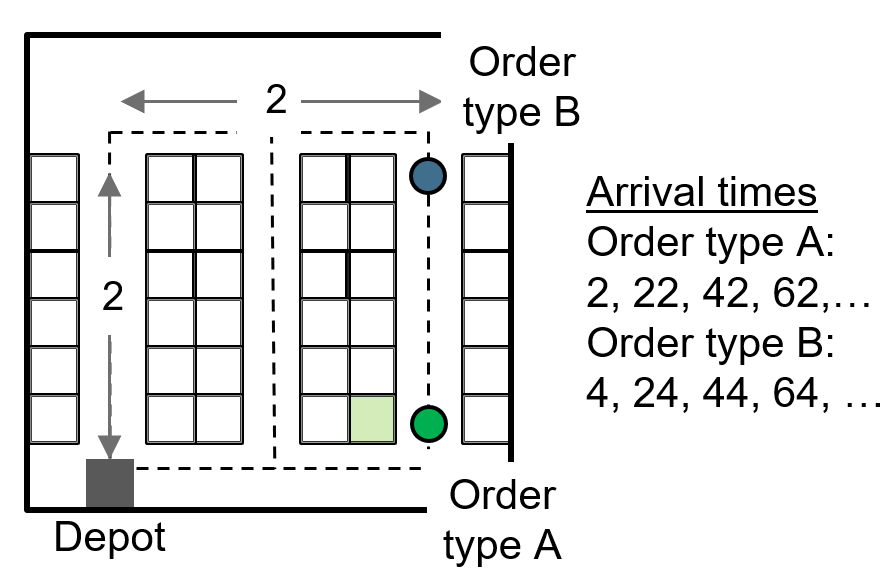}
\captionof{figure}{Example of strategic relocation for the \textit{turnover} objective}
\scriptsize{Consider $c=2$ and unit speed. Average turnover in CIOS is 10 (e.g. by picking pairs of A- and B-orders in batches). For any policy without strategic relocation, the average turnover cannot be better than 14.}
\end{minipage}
\end{figure}

While CIOSs optimize the picker routing for a given batch, there is a notable difference from common warehousing heuristics.  In CIOSs, the picker uses idle time to move towards the next picking location for an upcoming order -- a practice we term \textit{strategic relocation}. This occurs before the first pick for 39-60\% of orders in the \textit{makespan} objective and 41-62\% of orders in the \textit{turnover} objective (see Table~\ref{tab:strategicrelocation}), accounting for 8-19\% of the total completion time in both objectives. 

For the \textit{makespan} objective, time savings from strategic relocation are constrained by the warehouse's dimensions (the time to traverse the warehouse $(W+L)$) even for instances with many orders, as discussed in Lemma~\ref{lemma:strategic_reloc}. For the \textit{turnover} objective, strategic relocation can significantly improve the turnover time of \textit{every} individual order, leading to a noticeable enhancement of the overall objective (see example in Figure~\ref{fig:strategic_relocation}).

\begin{observation}[Anticipatory Routing] \label{obs:AnticipatoryBatching}
CIOSs make extensive use of strategic relocation, where the picker moves to the next order's picking location during idle time.
\end{observation}

\subsection{Discussion: Design of good online policies} ~\label{sec:good_online_algs}
In this section, we build upon the analysis from Section~\ref{sec:CIOS_analysis} to design effective online policies and estimate the potential for advanced anticipatory techniques.

Starting with the well-known \textit{Variable Time Window Batching (VTWB)} policy, which was extensively studied in the literature \citep{Vannieuwenhuyse2009, GilBorras2024}, we progressively refine its algorithmic elements based on the insights from the previous sections (see Figure~\ref{fig:avg_gaps_small} and Table~\ref{tab:gaps}). VTWB is a non-interventionist policy that employs simple FIFO-batching, S-shape routing, and waits at the depot until at least $N=2$ orders arrive (note that the value of $N=2$ equals the optimal amount of waiting in the examined warehouse based on the queuing-theoretical formulas of \citet{LeDuc2007} applied to our \textit{Basis} setting).  First, we replace the S-shape routing with \textit{optimal routing}  (cf. Section~\ref{sec:CIOS_routing}), then introduce \textit{optimal batching} in place of FIFO batching  (cf. Section~\ref{sec:CIOS_batching}). We further incorporate intervention, resulting in a myopic reoptimization policy with intervention \citep[cf.][]{Lorenza, Dauod2022, Giannikas2017}.  Given our findings on the high opportunity cost of strategic waiting and its dependence on anticipatory batching (Section~\ref{sec:CIOS_strategic_waiting}), we \textit{eliminate waiting}.

For the \textit{makespan} objective, we add a \textit{batch selection rule}, prioritizing batches with the largest number of orders and, in case of ties, those offering the highest savings (cf. Section~\ref{sec:CIOS_anticipation}). No such rule is applied for the \textit{turnover} objective, as reoptimization inherently selects batches to minimize additive order completion times. Finally, we allow the picker to move to the warehouse center during idle time, inspired by the concept of \textit{strategic relocation} (cf. Section~\ref{sec:CIOS_anticipation}). The resulting policy is termed \textit{Reopt*}. 

We evaluate the impact of of each algorithmic element by calculating the \textit{gap to the perfect anticipation result}. The gap of an algorithm ALG on instance $I$ compares $z^{\text{ALG}}(I)$ (the objective value achieved by ALG) to $z^*(I)$ (the optimal CIOS objective value) and is defined as:
\begin{align}
    \text{gap}(\text{ALG},I)=\frac{z^{\text{ALG}}(I)-z^*(I)}{z^*(I)}.
\end{align}
Observe that for \textit{turnover} objective, we compare the gaps of ALG’s average turnover time $z^{\text{turnover; ALG}}(I)$ to the \textit{upper bound} of $z^*(I)$, computed using the heuristic DP approach from Section~\ref{sec:imp_details}.

\begin{figure}
\centering
    \begin{subfigure}[b]{0.49\textwidth}
         \includegraphics[width=0.99\linewidth]{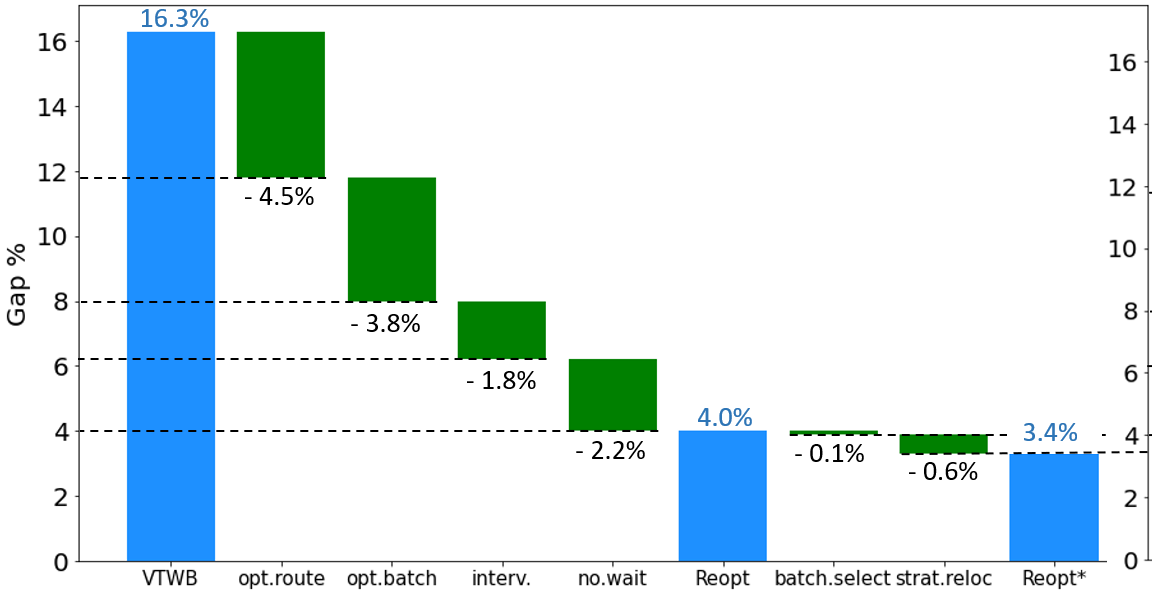}
    \label{fig:avg_gaps_small_makespan}
    \caption*{(a) Makespan minimization}
    \end{subfigure}
    ~ 
    \begin{subfigure}[b]{0.49\textwidth}
        \includegraphics[width=0.99\textwidth]{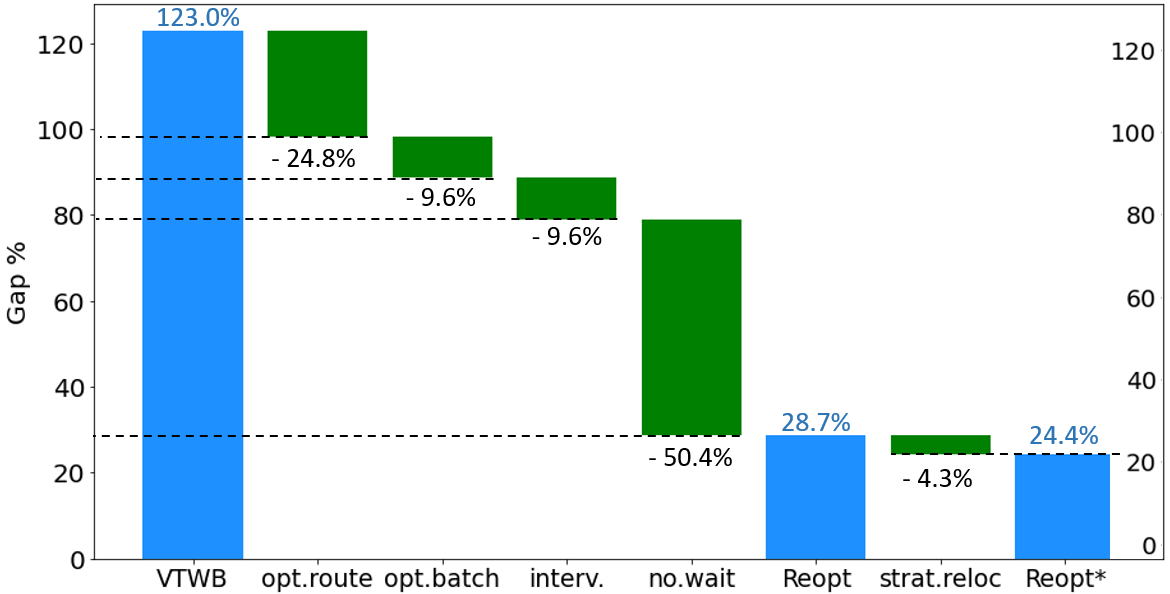}
        \caption*{(b) Average turnover minimization}
        \label{fig:avg_gaps_small_turnover}
    \end{subfigure}
    \caption{Development of the \textit{average} gap to the perfect anticipation result across various policies} \label{fig:avg_gaps_small}
\end{figure}

\begin{table}
\caption{Development of average- (worst-)  gap (\%) to the perfect anticipation result of online policies with algorithmic elements} \label{tab:gaps}
\centering
\scriptsize{
\begin{tabular}{@{}l@{}r|r|rrrrrr|r}
\toprule
& & & \multicolumn{6}{c}{Effect of algorithmic elements} &  \\
Obj. & \ Setting ($n^o=15$) & VTWB & opt.route & opt.batch & interv. & no.wait & batch.select  & strat.reloc & Reopt* \\
\midrule
\multirow{4}{*}{\rotatebox[origin=c]{90}{\makecell{Makesp. }}} & $LargeOrders\_c2\_r200$  & 19.9 (34.4) & -9.5 (-15.1) & -2.9 (-5.1) & -1.0 (+1.3) & -3.0 (-8.0) & -0.5 (-1.6) & 0.0 (0.0) & 3.0 (6.0) \\ 
& $SmallOrders\_c2\_r200$  & 10.9 (23.5) & -2.1 (-4.6) & -2.6 (-5.4) & -0.2 (-3.1)  & -2.8 (-1.3) & +0.2 (-0.1) & -0.5 (-1.0) & 2.7 (8.0)   \\ 
 & $SmallOrders\_c2\_r250$  & 12.4 (28.4) & -2.6 (-10.5) & -1.4 (-1.4) & -1.1 (+0.1) & -2.7 (-5.9) & 0.0 (-1.9) & -0.7 (0.0) & 3.9 (8.8) \\ 
& $SmallOrders\_c4\_r250$  & 22.1 (65.5) & -3.9 (-12.6) & -8.3 (-34.5) & -4.7 (-8.2) & -0.4 (+3.4) & 0.0 (-1.2)& -0.9 (0.0) & 3.9 (12.5) \\ 
    \midrule
\multirow{4}{*}{\rotatebox[origin=c]{90}{\makecell{Turnover }}} & $LargeOrders\_c2\_r200$  & 113.2 (151.5) & -42.7 (-27.9) & -5.1 (-17.5) & -8.7 (-9.5) & -34.2 (-56.9) & -- (--) & -3.2 (-10.6) & 19.2 (29.1) \\ 
& $SmallOrders\_c2\_r200$  & 147.6 (278.4) & -21.5 (-10.3) & -3.6 (-21.2) & -8.3 (-4.5) & -79.9 (-187.6) & -- (--) & -8.2 (-14.1) & 26.0 (40.5)  \\ 
& $SmallOrders\_c2\_r250$  & 107.9 (180.3) & -17.2 (-0.8)& -7.1 (-5.0) & -9.2 (+0.1) & -45.2 (-122.1) & -- (--) & -3.2 (-4.1) & 26.0 (48.8) \\ 
& $SmallOrders\_c4\_r250$  & 123.4 (181.8) & -17.8 (-26.0) & -22.5 (-10.6) & -12.5 (-4.9) & -42.2 (-85.9) & -- (--) & -2.5 (-1.3) & 26.4 (53.0)  \\ 
    \midrule    
\end{tabular}
}
\end{table}

The introduced algorithmic elements led to significant improvements in the basic VTWB policy. On average, gaps were reduced from 16.3\% to 3.4\% for the \textit{makespan} objective and from 123.0\% to 24.4\% for the \textit{turnover} objective across all settings (see Figure~\ref{fig:avg_gaps_small}).

While the positive effects of optimal routing and batching have been well-documented in the literature, the benefits of \textit{no-wait}, \textit{intervention}, and \textit{strategic relocation} have largely gone unnoticed \citep[cf.][]{Pardo2023}. Additionally, for the \textit{makespan} objective, tailored batch selection offers moderate improvements (0.1\% on average and up to 2\% in specific cases), especially for higher order arrival rates and large orders. 
Observe that the positive effect of intervention is stable across a large range of policies, see Appendix~\ref{sec:intervention}.

\subsection{Discussion: Quantifying the potential of advanced anticipation techniques} ~\label{sec:anticipationQ}
The last column of Table~\ref{tab:gaps} highlights the potential benefits of advanced anticipation, comparing the enhanced myopic reoptimization algorithm Reopt* with the perfect anticipation algorithm. Advanced anticipation could reduce the average \textit{makespan} by 2.7–3.9\%, with higher gains for larger order arrival rates and larger orders, where the potential for greater savings from batching is higher. 
For the \textit{turnover} objective, relative gains from advanced anticipation are more substantial, ranging from 19.2\% to 26.4\%, particularly for small orders. However, notice the comparatively small denominator (average order turnover) for these improvements. 
Overall, as outlined in Section~\ref{sec:good_online_algs}, simple adjustments can lead to significant improvements in basic myopic heuristics, even without the use of anticipation techniques.

\section{Conclusion} \label{sec:conclusion}
This paper investigates perfect-information policies for warehouse picking operations with dynamically arriving orders. The goal is to identify decision patterns that reduce order-picking costs and average order turnover, which can be integrated into practical warehousing policies. Additionally, the study uses perfect-information policies to assess the potential benefits of advanced anticipation techniques, such as contextual information and machine learning approaches. 

To achieve this, we developed a tailored dynamic programming algorithm (DP) for the \textit{Order Batching, Sequencing, and Routing Problem with release times (OBSRP-R)}, representing perfect-information operations in picker-to-parts warehouses. For the \textit{makespan} objective, which focuses on the picker’s time and associated wage costs, the DP is exact, making it the first exact solution approach for OBSRP-R in the literature. For the \textit{turnover} objective, the DP serves as an efficient heuristic and becomes exact if no waiting times are involved. Commercial solvers were incapable of solving even small instances optimally, even with extensive computational resources for parallelization, highlighting the necessity of our algorithms.

Our findings confirm the importance of using optimal batching and routing strategies over simpler heuristic rules like first-come-first-serve (FIFO) or S-shape routing. More importantly, our analysis generates actionable insights on how to improve both the makespan and the turnover objectives \textit{simultaneously}.

\textit{Insight 1. Implement intervention.} Intervention involves dynamically adjusting picking plans while the picker is `in the field', such as adding a new order to the cart or replacing an order whose items have not yet been picked with another. Although it requires investment in data and communication systems for real-time updates of the picker’s instructions, these adjustments can yield operational benefits. Intervention fully utilizes incoming order information reducing costs and speeding up operations.   Despite being highlighted by \citet{Giannikas2017} for its benefits, intervention has received little attention, notwithstanding its consistent positive impact across various heuristic approaches (see Section~\ref{sec:good_online_algs}).

\textit{Insight 2. Eliminate intentional waiting.} 
While brief periods of strategic waiting have shown some positive effects in FIFO policies \citep{LeDuc2007, Vannieuwenhuyse2009, Bukchin2012}, attempts to incorporate it into more advanced planning heuristics have largely failed \citep[cf.][]{GilBorras2024, Henn2012}. In fact, when intervention is allowed, strategic waiting becomes detrimental to both the makespan and the average order turnover. For the \textit{turnover} objective, eliminating strategic waiting yields the most significant improvement, cutting the average optimality gap by 50\% across all instances, with reductions as high as 188\% in one case. 
Our analysis in Section~\ref{sec:CIOS_strategic_waiting} explains these findings, showing that strategic waiting is effective only when it is brief, well-timed, and based on accurate order anticipation. This makes it challenging to integrate into non-anticipatory online policies. Ultimately, strategic waiting resembles an "all-in" gamble, with high opportunity costs if the waiting period is too long.

\textit{Insight 3. Relocate the picker to the 'gravity center' of anticipated future orders during idle time.} In practice, space constraints and proximity to sorting and packaging areas often prevent the depot from being centrally located, unlike in idealized warehouse designs. Allowing the picker to move to the 'order gravity center' during idle time can help offset this limitation. Similar strategies have been observed in other dynamic routing problems in the literature \citep{Bertsimas1993}.


\textit{Insight 4. Advanced anticipation offers limited leverage in picking operations -- focus data science efforts on upstream planning.}
Significant improvements in warehouse picking can be achieved using classic optimization methods, with only modest gains left for advanced anticipation, as shown by the perfect-information benchmark. For example, classic optimization in Figure~\ref{fig:avg_gaps_small} can improve the makespan (i.e., picker's time and wages) by about 13 percentage points, leaving a smaller 3-4\% improvement potential from advanced anticipation. Additionally, as discussed in Section~\ref{sec:CIOS_analysis}, the margin for error in anticipation is likely small in picking operations. Higher-impact decisions with more tolerance for anticipation errors are found in upstream areas like daily staffing, inventory levels, item repositioning, and zoning.

Perfect-information benchmarks offer useful insights on the limits and promises of advanced anticipation techniques. Future research is needed to investigate anticipation potential across further decisions in e-commerce warehousing. Future studies may improve the scalability of online algorithms, especially those that allow for time-consuming optimal batching and routing. 
Additionally, integrating order picking and delivery into a single optimization framework could yield further efficiencies.

\section*{Acknowledgments}\label{sec:acknowledgements}
We would like to extend our gratitude to the Government of Quebec (Ministère des Relations internationales et de la Francophonie) and BayFor for their support. We also acknowledge Compute Canada for providing the computational resources necessary for this research. 

 \let\oldbibliography\thebibliography
 \renewcommand{\thebibliography}[1]{%
    \oldbibliography{#1}%
    \baselineskip14pt 
    \setlength{\itemsep}{5pt}
 }
\bibliography{AGV_picking}
\newpage

\section*{Appendix}

\subsection{Mixed-integer linear programming formulations for OBSRP-R}\label{sec:MIPs}
In this section, we present \textit{mixed-integer programming (MIP)} formulations for OBSRP-R with the makespan (Section~\ref{sec:MIP_cost}) and turnover (Section~\ref{sec:MIP_turnover}) objective.

In addition to the previous notation summarized in Table~\ref{tab:notation}, we introduce in the MIP formulations a $n^i\times n^o$ \textit{adjacency matrix} $A=(a_{sj})_{s\in S, j \in [n^o]}$ for which each entry $a_{sj} \in A$ takes the value 1 if item $s$ belongs to order $o_j$, and the value 0, else.  
Let $K$ denote an upper bound for the number of batches in an optimal solution. 

For both objectives, we use a three-index formulation 
for the order batching, sequencing, and picker routing problem without release times.  

We use the following decision variables:

\begin{itemize}
\item For each batch index  $k\in [K]$, and each pair of locations $i\in S\cup \{l_d\}, l \in S \cup \{l_d\}, i \neq l$, in the of set picking locations or the depot, the binary variables $y_{kil}  $ indicate if location $l$ is visited within the $k^{\text{th}}$ batch, and if it is visited directly after location $i$ in this batch (then, $y_{kil}=1$), or not (then, $y_{kil}=0$).
\item For each batch index  $k\in [K]$ and every $j\in[n^o]$, the binary variable  $x_{kj} $ indicates if order $o_j$ is in the $k^{\text{th}}$ batch (then, $x_{kj}=1$), or not (then, $x_{kj}=0$). 
\item For each item $s\in S$, the continuous variable $t_{s}$ represents the completion time $C(s)$ of item $s$.   
\item For each batch index $k\in [K+1]$, the continuous variable $t_{l_d}^k$ represents the time the picker leaves the depot $l_d$ to start the $k^{\text{th}}$ batch.  
\end{itemize}

\subsubsection{MIP formulation for the makespan objective}\label{sec:MIP_cost}

For the makespan objective, we can use a customized upper bound $K$ for the number of batches, which is given by Proposition~\ref{prop:bound_for_nbr_batches}.
\begin{proposition}\label{prop:bound_for_nbr_batches}
For every instance of OBSRP-R with the objective of minimizing the makespan, 
there exists an optimal solution that forms at most $K$ batches, with
\begin{align}
K:=\lfloor n^o \cdot \frac{2}{c+1}\rfloor \label{eq:K_UB_batches}
\end{align}
\end{proposition}
\textit{Proof}. 
Consider a solution $\hat{\sigma}$ with $\pi^{\text{batches}}(\hat{\sigma})=(\hat{B}_1,...,\hat{B}_l, \hat{B}_{l+1},...,\hat{B}_f)$, such that for two \textit{consecutive} batches $\vert \hat{B}_l \vert + \vert \hat{B}_{l+1} \vert \leq c $. Then $\hat{\sigma}$ is weakly dominated by the solution $\tilde{\sigma}$ that conserves the visiting sequence of items $\pi(\hat{\sigma})$, and the sequence of batches $\pi^{\text{batches}}(\hat{\sigma})$, except from replacing $\hat{B}_l, \hat{B}_{l+1}$ by a single batch $\tilde{B}$, with $\pi^{\tilde{B}}=(\pi^{\hat{B}_l},\pi^{\hat{B}_{l+1}})$. This fact follows from the triangle inequality of the considered distance metric: transitioning directly from the last item of $\pi^{\hat{B}_l}$ to the first item of $\pi^{\hat{B}_{l+1}}$ is at least as fast as taking the detour through the depot. Thus, even with release times for the orders, at most the waiting time for one of the items in ${B}_{l+1}$ may increase, which has no consequences on the solution's completion time, and the overall makespan of the solution can not deteriorate.   

By definition, an optimal solution with the smallest number of batches $\tilde{K}$ cannot be weakly dominated by a solution with fewer batches, and from the above, we conclude that each pair of consecutive batches in such a solution may count together at least $c+1$ orders. In case $2n^o \mod (c+1)=0$, every solution having this property has at most as many batches as a solution that partitions the $n^o$ orders into pairs of batches with \textit{exactly}  $c+1$ orders. There are $\frac{n^o}{c+1}$ such pairs of batches which can be formed, thus $\frac{2\cdot n^o}{c+1}$ individual batches. In the case where $2n^o \mod (c+1)>0$, there must be even one pair of consecutive batches with more than $c+1$ orders, and the maximum amount of formed batches is $K$ as in (\ref{eq:K_UB_batches}).
\qed

We propose the following MIP formulation:

\begingroup
\allowdisplaybreaks
\begin{align}
\text{Minimize} \max_{k\in [K+1]} \{t_{l_d}^{k} \}   \label{eq:MIP_obj} \\
\text{s.t.} & \nonumber \\
\sum\limits_{k=1}^K \sum\limits_{i \in S \cup \{l_d\} ,i\neq l}  y_{kis}=1 \qquad & \forall s \in S  \label{eq:MIP_flow1}\\
 \sum\limits_{i \in S }  y_{kil_d}\leq1 \qquad & \forall k \in [K]  \label{eq:MIP_flow1_depot}\\
    \sum\limits_{i \in S \cup \{l_d\}, i \neq s} y_{kis} \leq \sum\limits_{i \in S \cup \{l_d\}, i \neq l} y_{ksi} \qquad & \forall s \in S, \forall k \in [K] \label{eq:MIP_flow2} \\
     \sum\limits_{k =1 }^K\sum\limits_{s \in S } y_{ksl_d} =  \sum\limits_{k =1 }^K\sum\limits_{s \in S } y_{kl_ds} \qquad &  \label{eq:MIP_flow2_depot} \\ 
   \sum\limits_{s \in S }y_{kl_ds}\leq \sum\limits_{s \in S }y_{(k-1)l_ds}  \qquad & \forall k \in \{2,3,...K\} \label{eq:symmetrie_breaking}   \\
x_{kj} \geq a_{sj} \sum\limits_{i \in S \cup \{l_d \}} y_{kis} \qquad & \forall j \in [n^o], \forall k \in [K], \forall s \in S \label{eq:MIP_orderbatch}\\   
\sum\limits_{k=1}^K x_{kj} \leq 1 \qquad & \forall j \in [n^o] \label{eq:MIP_ordercover} \\
\sum\limits_{j =1}^{n^o} x_{kj} \leq c \qquad & \forall k \in [K] \label{eq:MIP_batchsize} \\
t_s \geq t_i -  M(1-\sum\limits_{k=1}^K y_{kis}) + \frac{1}{v} \cdot d(i,s)+  t^{p} \qquad & \forall s \in S, \forall i \in S, s \neq i \label{eq:MIP_MTZ1} \\
   t_s \geq t_{l_d}^k -  M(1- y_{kl_ds}) + \frac{1}{v} \cdot d(l_d,s)+  t^{p} \qquad & \forall s \in S, \forall k \in [K] \label{eq:MIP_MTZ2} \\
   t_{l_d}^k \geq t_s -  M(1- y_{(k-1)sl_d}) + \frac{1}{v} \cdot d(s,l_d) \qquad & \forall s \in S, \forall k \in \{2,3,...,K+1\} \label{eq:MIP_MTZ3} \\
   t_s  \geq r_j \cdot a_{sj} + t^{p} \qquad & \forall s \in S , \forall j \in [n^o] \label{eq:MIP_release_times}\\ 
y_{kil} \in \{0,1\}  \qquad & \forall k \in [K], \forall i,l \in S \cup \{l_d\}, i \neq l \label{eq:MIP_y}\\
x_{kj} \in \{0,1\} \qquad & \forall  k \in [K], \forall j \in [n^o] 
\label{eq:MIP_x}\\
t_i \geq 0 \qquad & \forall i \in S \label{eq:MIP_t}\\
  t_{l_d}^k \geq 0 \qquad & \forall k \in [K+1]\label{eq:MIP_t2} 
\end{align}
\endgroup
The objective function (\ref{eq:MIP_obj}) minimizes the total completion time. 
The constraints (\ref{eq:MIP_flow1}) - (\ref{eq:MIP_flow2_depot}) are flow constraints, which also ensure that every picking location is visited. Constraints (\ref{eq:symmetrie_breaking}) are symmetry-breaking constraints.
Constraints (\ref{eq:MIP_orderbatch}) and (\ref{eq:MIP_ordercover}) assign orders to batches, such that all items that belong to the same order are placed in the same batch. Constraints (\ref{eq:MIP_batchsize}) define the batch capacities. Constraints (\ref{eq:MIP_MTZ1})-(\ref{eq:MIP_MTZ3}) are \textit{Miller-Trucker-Zemlin (MTZ)} subtour elimination constraints and at the same time define the completion time of the items. 
The following large parameter $M$ is used within these constraints:
\begin{align}
    M:=r_{n^o}+\frac{1}{v}\cdot (2\cdot L + (a+1)\cdot W) \cdot \lceil \frac{n^o}{c} \rceil + n^i \cdot t^p + \frac{1}{v} \cdot (L+W) \label{eq:Big_M_def}
\end{align}
as it is an upper bound for the expressions $t_i+\frac{1}{v} d(i,s)+t^p$ and $t_{l_d}^k+\frac{1}{v} d(l_d,s)+t^p$ for all $s \in S\cup\{l_d\}$ and thus guarantees that the right-hand side of the constraints (\ref{eq:MIP_MTZ1})-(\ref{eq:MIP_MTZ3}) are non-positive in case $M$ is multiplied by 1. To see this, note that the problem turns into a static problem after the arrival of the last order $r_{n^o}$, whose optimal completion time is bounded from above by the time to pick $\lceil \frac{n^o}{c} \rceil $ arbitrarily formed batches by traversing the warehouse completely in an S-shape motion of at most $ \frac{1}{v} \cdot (2\cdot L + (a+1)\cdot W) $ time, see also Lemma 1 in \citet{Lorenza}. Thus the first two summands of $M$ form an upper bound for $t_i$ and $t_{l_d}^k$ in any optimal solution, and similarly, the last summand of M bounds from above the traversal time $\frac{1}{v} d(i,s)$ for all $ i,s \in S\cup\{l_d\}$.  
Constraints (\ref{eq:MIP_release_times}) make sure that the release times are respected and (\ref{eq:MIP_y})-(\ref{eq:MIP_t2}) define the nature of the variables. \\

\subsubsection{MIP formulation for the turnover objective}\label{sec:MIP_turnover}
 In the case of the turnover objective, one additional group of variables is needed:
\begin{itemize}
    \item For each order $o_j, j \in [n^o]$, the continuous variable $t^o_{j}$ represents the completion time $C(o_j)$ of order $o_j$.  
\end{itemize}
Furthermore, the upper bound on the number of batches $K$ from Proposition~\ref{prop:bound_for_nbr_batches} does not hold in this case, thus, $K:=n^o$ is used as an upper bound.

The objective in this case writes as follows:
\begin{align}
    \text{Minimize } \sum\limits_{j=1}^{n^o} \frac{1}{n^o} t_j^o - \sum\limits_{j=1}^{n^o} \frac{1}{n^o} r_j,
\end{align}
where we remember that $ \sum\limits_{j=1}^{n^o} \frac{1}{n^o} r_j$ is a constant. 

In addition to all constraints (\ref{eq:MIP_flow1})-(\ref{eq:MIP_t2}) from the previous section, the following must hold to define the completion times of the orders:
\begin{align}
    t_j^{o} \geq t^{k+1}_{l_d}-N(1-x_{kj}) \qquad & \forall j \in [n^o], \forall k \in [K+1]\\
    t_j^o \geq 0 \qquad & \forall j \in [n^o]
\end{align}
Similarly to (\ref{eq:Big_M_def}) the large parameter $N$ receives the value
\begin{align}
    N:=r_{n^o}+\frac{1}{v}\cdot (2\cdot L + (a+1)\cdot W) \cdot \lceil \frac{n^o}{c} \rceil + n^i \cdot t^p.
\end{align}

\subsection{Correctness of the formulated state graph}\label{sec:dp_correct}
The following proposition claims a one-to-one correspondence between the paths in the constructed state graph and the feasible solutions of OBSRP-R. 
\begin{proposition}
For a particular OBSRP-R instance, 
each path from the initial state to the terminal state in the constructed state graph corresponds to exactly one feasible solution for this instance and vice-versa. 

\end{proposition}\label{prop:1-1correspondence}
\begin{proof}
Given that the first entry in each state at stage $k\in[n^i]$ represents one picked item, selected in the previous transition either from set $S^{\text{batch}}$ or an order from sequence $O^{\text{pend}}$ which contain only outstanding items, the sequence of these entries in a path of states forms a visiting sequence of picking locations $\pi$ of Hamiltonian nature. A partitioning of this sequence into batches can be derived by closing one batch after each visit of a batch-completion state and opening a new batch with the picking item of the following state. By construction, the resulting batches form a disjoint partition of all $n^o$ orders of the instance, and each batch contains at most $c$ orders.

The other way around, each feasible solution translates to exactly one path of states connected by transitions from Table~\ref{tab:pcart_transitions} of Section~\ref{sec:problem_desc} by construction. 
\end{proof}

\subsection{Detailes on the Bellman equations}~\label{sec:details_bellman}
The \textit{immediate transition costs} represent the amount of \textit{time} necessary for the transition between states, as explained in Section~\ref{sec:cost_labeling}. They are important for the DP algorithm of the makespan objective (see Section~\ref{sec:cost_labeling}, and the DP algorithm of the turnover objective (see Section~\ref{sec:turnover_heuristic}).
In the following, we provide the detailed formulas. 

If the transition results in a state $\Theta_{k+1}$ with last-picked item $s_{k+1}$, it is composed of the maximum of the picker's walking time, and the waiting time for the release of $s_{k+1}$ -- plus the picking time $t^p$. If $\Theta_{k+1}$ is a batch-completion state, the walking time from $s_{k+1}$ to the depot $l_d$ is added. 
To account correctly for the waiting time, $g(\Theta_k,x_k)$ depends on the current time at the departure state, i.e. the time value $\Omega^{\text{*,cost}}(\Theta_k)$.  Let denote by $\mathcal{A}^{\text{batch-comp}}$ the subset of batch-completion states.
Altogether, the immediate costs of the transition from 
state $\Theta_k$, with value $\Omega^{\text{*,cost}}(\Theta_k)$, towards a state $\Theta_{k+1}=f(\Theta_k,x_k)$ is: 
\begin{align}
    & g^{\text{cost}}(\Theta_k,x_k) \nonumber \\ =& \begin{cases}
        \max\{\frac{d(s_k,s_{k+1})}{v}; r(s_{k+1})-\Omega^{*,\text{cost}}(\Theta_k)\}+t^p & \text{if }k< n^i,\Theta_k \notin \mathcal{A}^{\text{batch-comp}} \text{ and }\Theta_{k+1} \notin \mathcal{A}^{\text{batch-comp}} \\
        \max\{\frac{d(l_d,s_{k+1})}{v}; r(s_{k+1})-\Omega^{*,\text{cost}}(\Theta_k)\}+t^p & \text{if }k< n^i,\Theta_k \in \mathcal{A}^{\text{batch-comp}} \text{  and }\Theta_{k+1} \notin \mathcal{A}^{\text{batch-comp}}  \\
        \max\{\frac{d(s_k,s_{k+1})}{v}; r(s_{k+1})-\Omega^{*,\text{cost}}(\Theta_k)\}+t^p+\frac{d(s_{k+1},l_d)}{v}& \text{if }k< n^i,\Theta_k \notin \mathcal{A}^{\text{batch-comp}} \text{ and }\Theta_{k+1} \in \mathcal{A}^{\text{batch-comp}} \\
        \max\{\frac{d(l_d,s_{k+1})}{v}; r(s_{k+1})-\Omega^{*,\text{cost}}(\Theta_k)\}+t^p+\frac{d(s_{k+1},l_d)}{v}& \text{if }k<n^i,\Theta_k \in \mathcal{A}^{\text{batch-comp}}\text{ and }\Theta_{k+1} \in \mathcal{A}^{\text{batch-comp}} \\
        0 &\text{if }k=n^i
    \end{cases} \label{eq:immediate_cost}
\end{align}

\subsection{Proof of dominance rules}
In this section, we consider an arbitrary instance $I$ for OBSRP-R and show that the dominance rule claimed in  Section~\ref{sec:dominancerules} are valid for both objective functions - makespan and turnover minimization. 

For the proofs, we need the formulas for reconstructing an optimal \textit{picking schedule} (see discussion in Section~\ref{sec:obj_fun}) from a given picking sequence $\pi$ for instance $I$, and its associated batches $\pi^{\text{batches}}$.
The constructed schedule aims for each item to be collected at the \textit{earliest possible time} (which accounts for the release times of the orders).  This is achieved by applying the formulas from equations (\ref{eq:pcart1}) to (\ref{eq:pcart3}). For convenience, let denote the $i^{th}$ item in $\pi$ as $\pi[i]$.

\begin{align}
C(\pi[1])= &\max\{\frac{1}{v}\cdot d(l_d,\pi[1]),r(\pi[1])\}+t^p \label{eq:pcart1}\\
C(\pi[i+1])  = &
\begin{cases}
\max\{C(\pi[i])+\frac{1}{v}\cdot d(\pi[i],\pi[i+1]),r(\pi[i+1])\}+t^p,\\ \qquad  \qquad \text{if }\pi[i], \pi[i+1] \text{ belong to the same batch} \\ 
\max\{C(\pi[i])+\frac{1}{v}\cdot d(\pi[i],l_d) + \frac{1}{v}\cdot d(l_d,\pi[i+1]),r(\pi[i+1])\}+t^p, 
\\ \qquad \qquad \text{if } \pi[i], \pi[i+1] \text{ belong to distinct batches }  \label{eq:pcart3}
\end{cases} \nonumber \\
&  \forall i\in[n^i-1]  
\end{align}

The second case of equation (\ref{eq:pcart3}) occurs when the picker, having completed a batch with item $\pi[i]$, returns to the depot for unloading. Subsequently, the picker starts a new batch using a fresh cart, beginning with the collection of item $\pi[i+1]$.
\subsubsection{Proof of Proposition~\ref{prop:dom_in_batch_1} }\label{sec:proof_dom_in_batch_1}

We use Lemma~\ref{prop:neighbor_items} to define a set $D_1$ of feasible solutions for OBSRP-R instance $I$ that are weakly dominated by other feasible solutions. In the second step, Lemma~\ref{lem:association_trans_path_1} states that every path from the initial state to the terminal state in the constructed state graph - refered to as a \textit{complete path} in the following, which involves a transition $x_k, k\in[n^i-1]$, from a label of state $\Theta_k$ with value $\Omega^{\text{*,cost}}$ as described by Proposition~\ref{prop:dom_in_batch_1}, is associated to a solution in set $D_1$.  

\begin{lemma}~\label{prop:neighbor_items}
Consider any feasible solution $\hat{\sigma}$ with a visiting sequence of picking locations $\hat{\pi}$  and the respective sequence of batches $\hat{\pi}^{\text{batches}}$ such that some item $\hat{\pi}[i+1]$ has a late release date:
\begin{align}
r(\hat{\pi}[i+1]) \geq C(\hat{\pi}[i])  + \frac{1}{v} \cdot (d(\hat{\pi}[i],s^{\dagger})+d(s^{\dagger},\hat{\pi}[i+1])) +t^{p}\label{eq:hilf0}
\end{align}
for some fixed $s^{\dagger} \in o_{\hat{j}}$, such that $s^{\dagger}=\hat{\pi}[i+k]$ for some $k>1$, i.e. $s^{\dagger}$ is picked \textit{after} $\hat{\pi}[i+1]$ in $\hat{\sigma}$, and, such that the order $o_{\hat{j}}$ is already commenced at the picking of item $\hat{\pi}[i+1]$, i.e. some item of $o_{\hat{j}}$ is picked before $\hat{\pi}[i+1]$.

Let decompose the visiting sequence of $\hat{\sigma}$ as follows: $\hat{\pi}=(\hat{\pi}^I,\hat{\pi}[i+1],\hat{\pi}^{II} )$, i.e. $\hat{\pi}^I$ and $\hat{\pi}^{II}$ denote the subsequences \textit{before} and  \textit{after} the visit of location  $\hat{\pi}[i+1]$, respectively. 

Then the following solution $\tilde{\sigma}$ weakly dominates solution $\hat{\sigma}$:
\begin{itemize}
 \item $\pi^{\text{batches}}(\tilde{\sigma})=\hat{\pi}^{\text{batches}}$
\item $\pi(\tilde{\sigma})=(\hat{\pi}^I, s^{\dagger},\hat{\pi}[i+1], \hat{\pi}^{II}\setminus s^{\dagger})$, where $\hat{\pi}^{II}\setminus s^{\dagger}$ is the sequence $\hat{\pi}^{II}$ after removing item $s^{\dagger}$.\\
In other words, solution $\tilde{\sigma}$ collects item $s^{\dagger}$ immediately before item $\hat{\pi}[i+1]$.
\end{itemize}
\end{lemma}
\begin{proof}
By construction, $\tilde{\sigma}$ is a feasible solution. What remains to show, is that $z^{obj}(\tilde{\sigma})\leq z^{obj}(\hat{\sigma})$, where $obj\in\{$makespan, turnover\}, depending on the objective.

Let compute the \textit{schedule} of both solutions $\tilde{\sigma}$ and $\hat{\sigma}$ as described in (\ref{eq:pcart1})-(\ref{eq:pcart3}). Since some items of order $o_{\hat{j}}$ are picked before item $\hat{\pi}[i+1]$ and since $s^{\dagger}$ also belongs to order $o_{\hat{j}}$:
\begin{align}
r(s^{\dagger})\leq C(\hat{\pi}[i]) \label{eq:hilf1}
\end{align}

In both $\tilde{\sigma}$ and  $\hat{\sigma}$, items $\hat{\pi}[i],\hat{\pi}[i+1]$ and $s^{\dagger}$ are in the same batch $B(\hat{\sigma},o_{\hat{j}})$ and the visiting sequences of the first $i$ items are the same in $\tilde{\sigma}$ and $\hat{\sigma}$.
Thus:
\begin{align}
    C(\hat{\pi}[l],\tilde{\sigma})= C(\hat{\pi}[l],\hat{\sigma}) \text{ for all }l\leq i
\end{align}
and 
\begin{align}
 C(o_j,\tilde{\sigma})= C(o_j,\hat{\sigma}) \text{ for all } o_j \text{ scheduled in batches prior to } B(\hat{\sigma},o_{\hat{j}}) \text { in }\hat{\pi}^{\text{batches}}
\end{align}

By the triangle inequality of the distance metric, we furthermore have: $C(s^{\dagger},\tilde{\sigma})\leq  C(s^{\dagger},\hat{\sigma})$.

Observe that in solution $\tilde{\sigma}$, item $\hat{\pi}[i+1]$ is picked immediately after visiting location $s^{\dagger}$. Furthermore,  By applying (\ref{eq:pcart1})-(\ref{eq:pcart3}), the completion time of item $\hat{\pi}[i+1]$ in solution $\tilde{\sigma}$ is:
\begin{align}
C(\hat{\pi}[i+1], \tilde{\sigma}) = & \max\{C(s^{\dagger})+\frac{1}{v} d(s^{\dagger},\hat{\pi}[i+1] ),r(\hat{\pi}[i+1])\}+ t^{p} \\
= &\max\{\max\{C(\hat{\pi}[i])+\frac{1}{v} d(\hat{\pi}[i],s^{\dagger}),r(s^{\dagger})\}+ t^{p}+\frac{1}{v} d(s^{\dagger},\hat{\pi}[i+1]), r(\hat{\pi}[i+1]) \} +t^{p}  \label{eq:D1_proof_1 }\\
= &\max\{C(\hat{\pi}[i])+\frac{1}{v} d(\hat{\pi}[i],s^{\dagger})+ t^{p}+\frac{1}{v} d(s^{\dagger},\hat{\pi}[i+1]), r(\hat{\pi}[i+1]) \}+t^{p}\label{eq:D1_proof_2 }\\
\leq & r(\hat{\pi}[i+1])+t^p,
\end{align}
where the equality (\ref{eq:D1_proof_1 }) follows from (\ref{eq:hilf1}) and the subsequent inequality (\ref{eq:D1_proof_2 }) follows from (\ref{eq:hilf0}). 

Similarly, by applying (\ref{eq:pcart1}) and (\ref{eq:pcart3}), the completion time of item $\hat{\pi}[i+1]$ in solution $\hat{\sigma}$ cannot be smaller, since:
\begin{align}
C(\hat{\pi}[i+1], \hat{\sigma}) = \max\{C(\hat{\pi}[i])+\frac{1}{v}\cdot d(\hat{\pi}[i],\hat{\pi}[i+1]), r(\hat{\pi}[i+1]) \} +t^p \geq  r(\hat{\pi}[i+1])+t^p
\end{align}

In other words:
\begin{align}
C(\hat{\pi}[i+1], \tilde{\sigma})\leq C(\hat{\pi}[i+1], \hat{\sigma})
\end{align}

After applying (\ref{eq:pcart3}) to compute the completion times of the remaining items and observing that the distances are metric and that the sequences of the visiting locations after picking $\hat{\pi}[i+1]$ coincide in $\tilde{\sigma}$ and $\hat{\sigma}$, but $\tilde{\sigma}$ excludes the location of the already picked item $s^{\dagger}$, we receive that $C(\tilde{\pi}[l],\tilde{\sigma})\leq C(\tilde{\pi}[l],\hat{\sigma})$ for all $l\geq i+1$, and thus by (\ref{eq:order_comp}) that $C(o_j,\tilde{\sigma})\leq C(o_j,\hat{\sigma})$ for all orders $o_j$ in batch $B(\hat{\sigma},o_{\hat{j}})$ and batches scheduled after this batch. By definitions (\ref{eq:obj_pushcart_makespan}) and (\ref{eq:obj_turnover}) we receive that $z^{\text{makesp}}(\tilde{\sigma})\leq z^{\text{makesp}}(\hat{\sigma})$ or $z^{\text{turnover}}(\tilde{\sigma})\leq z^{\text{turnover}}(\hat{\sigma})$, depending on the objective.

 \end{proof}
We denote the set of the weakly dominated feasible solutions $\hat{\sigma}$ described by Lemma~\ref{prop:neighbor_items} as $D_1$.

In Proposition~\ref{prop:1-1correspondence}, we have established the one-to-one correspondence between feasible solutions of an instance $I$ and the complete paths of the corresponding state graph. Furthermore, every state $\Theta_k$ of the graph with a specific value $\Omega^{\text{*,cost}}(\Theta_k)$ represents exactly one sub-path to reach $\Theta_k$ from the initial state.   Lemma~\ref{lem:association_trans_path_1} uses this relation.

\begin{lemma}~\label{lem:association_trans_path_1}
Consider a feasible transition $x_k\in X(\Theta_k)$ from a state $\Theta_k=(s_k,m^o,S^{\text{batch}},O^{\text{pend}})$ with $\vert S^{\text{batch}} \vert \geq 1$, and an associated value $\Omega^{*,\text{cost}}(\Theta_k)$, that dictates a next picking location $s_{k+1}\in o_j, o_j\in O^{\text{pend}}$ with the following property for some  $s^{\dagger}\in S^{\text{batch}}$
\begin{align}
r_j \geq  \Omega^{\text{*,cost}}(\Theta_k) + \frac{1}{v} \cdot (d(s_k,s^{\dagger})+d(s^{\dagger},s_{k+1})) +t^{p}  \label{eq:exactD1}
\end{align}
Then, any complete path that involves $x_k$ belongs to the class of dominated solutions $D_1$ described by Lemma~\ref{prop:neighbor_items}.
\end{lemma}

\begin{proof}
First, note that the items $s_{k+1}$ and $s^{\dagger}$ of Lemma~\ref{lem:association_trans_path_1} correspond to items $\hat{\pi}[i+1]$ an $s^{\dagger}$ in Lemma~\ref{prop:neighbor_items}, respectively, and that $r(s_{k+1})=r_j$ given that $s_{k+1}\in o_j$.  
Let $P^{\text{sub}}$ be the path of transitions from the initial state to state $\Theta_k$ that is associated to the value $\Omega^{\text{*,cost}}(\Theta_k)$. Consider any complete path that uses $P^{\text{sub}}$ as a sub-path, followed by $x_k$, denote by $\hat{\sigma}$ the corresponding feasible solution as by Proposition~\ref{prop:1-1correspondence}. The completion time of $s_k$ in  $\hat{\sigma}$  is $C(s_k,\hat{\sigma})=\Omega^{\text{*,cost}}(\Theta_k)$ by definition.  Because of (\ref{eq:exactD1}), property (\ref{eq:hilf0}) holds for $\hat{\sigma}$ and therefore, $\hat{\sigma}$ belongs to the set of weakly dominated solutions $D_1$. 

 \end{proof}

Finally, note that the right-hand side of (D1) from Proposition~\ref{prop:dom_in_batch_1} is an \textit{upper bound} of the right-hand side of (\ref{eq:exactD1}), since the distance between any two locations in the warehouse is bounded by its dimensions, $L+W$. Thus, the conditions stated in Proposition~\ref{prop:dom_in_batch_1} imply Lemma~\ref{lem:association_trans_path_1}, which concludes the proof.

\subsubsection{Proof of Proposition~\ref{prop:dom_in_batch_2}} \label{sec:proof_dom_in_batch_2}
We proceed along similar lines as in Section~\ref{sec:proof_dom_in_batch_1}. For an arbitrary instance $I$ of OBSRP-R, we use Lemma~\ref{prop:neighbor_items_2} to define the set $D_2$ of feasible solutions for $I$ that are weakly dominated by another
feasible solution. In the second step, Lemma~\ref{lem:association_trans_path_2} associates every complete path that uses $x_k, k \in [n^i -1]$ from state $\Theta_k$ with value $\Omega^{\text{*,cost}}(\Theta_k)$ as described by Proposition~\ref{prop:dom_in_batch_2}, to a solution in set $D_2$.

\begin{lemma}~\label{prop:neighbor_items_2}
Consider any feasible solution $\hat{\sigma}$ with a visiting sequence of picking locations $\hat{\pi}$  and the respective sequence of batches $\hat{\pi}^{\text{batches}}$ such that:
\begin{itemize}
\item Some item $\hat{\pi}[i+1]$ that belongs to some batch $\hat{B}_l\in \hat{\pi}^{\text{batches}}$ has a late release date:
\begin{align}
r(\hat{\pi}[i+1]) \geq C(\hat{\pi}[i])  + \frac{1}{v} \cdot (d(\hat{\pi}[i],l_d) + d(l_d,\hat{\pi}[i+1])) \label{eq:hilf0b}
\end{align}
\item This batch $\hat{B}_l=\hat{B}_l^I\cup \hat{B}_l^{II}$ can be partitioned into two sets of orders $\hat{B}_l^I$ and $\hat{B}_l^{II}$ such that
\begin{itemize}
\item $\hat{B}_l^I$ and $\hat{B}_l^{II}$ are picked subsequently in $\hat{\pi}^{\hat{B}_l}=(\hat{\pi}^{\hat{B}_l, I}, \hat{\pi}^{\hat{B}_l, II})$,
\item item $\hat{\pi}[i+1]$ is the first item in the sequence $\hat{\pi}^{\hat{B}_l, II}$.
\end{itemize}
\end{itemize}

Then the following solution $\tilde{\sigma}$ weakly dominates solution $\hat{\sigma}$:
\begin{itemize}
 \item Orders of $\hat{B}_l^I$ and $\hat{B}_l^{II}$ are picked in separate batches, whereas the remaining batches and their sequence remain the same: $\pi^{\text{batches}}(\tilde{\sigma})=(\hat{B}_1, \ldots, \hat{B}_{l-1}, \hat{B}_l^I, \hat{B}_l^{II}, \hat{B}_{l+1}, \ldots)$
\item $\pi(\tilde{\sigma})=\hat{\pi}$, i.e. the sequences of picking locations coincide.
\end{itemize}



\end{lemma}
\begin{proof}
The proof proceeds along the same lines as the proof of Lemma~\ref{prop:neighbor_items}. By construction, $\tilde{\sigma}$ is a feasible solution. What remains to show, is that $z^{obj}(\tilde{\sigma})\leq z^{obj}(\hat{\sigma})$, where $obj\in\{$makespan, turnover\}, depending on the objective. 

We apply (\ref{eq:pcart3}) and notice that $C(\hat{\pi}[i+1], \tilde{\sigma}) \leq C(\hat{\pi}[i+1], \hat{\sigma})=r(\hat{\pi}[i+1]) + t^p$. By the recursive nature of equations (\ref{eq:pcart1})-(\ref{eq:pcart3}), all following items in $\hat{\pi}$ conserve their completion time and by (\ref{eq:order_comp}) order completion times are identical in $\tilde{\sigma}$ and $\hat{\sigma}$. The conclusion of the proof follows by the definition of the objective functions (\ref{eq:obj_pushcart_makespan}) and (\ref{eq:obj_turnover}).
 \end{proof}

Let denote the set of weakly dominated solutions $\hat{\sigma}$ introduced in Lemma~\ref{prop:neighbor_items_2} as $D_2$.  Lemma~\ref{lem:association_trans_path_2} associates each path through the state graph that involves a transition $x_k$ described by Proposition~\ref{prop:dom_in_batch_2} to a weakly dominated solution, since, again, the right-hand side of (D2) is an \textit{upper bound} of the right-hand side of (\ref{eq:dom_rule_2_proof}), and this closes the proof.

\begin{lemma}~\label{lem:association_trans_path_2}
Consider a feasible transition $x_k\in X(\Theta_k)$ from a state $\Theta_k=(s_k,m^o,\{\},O^{\text{pend}})$ with $m^o \geq 1$, that dictates a next picking location $s_{k+1}\in o_j, o_j\in O^{\text{pend}}$ with the following property  
\begin{align}
r_j \geq  \Omega^{\text{time}}(\Theta_k) + \frac{1}{v} (d(s_k,l_d) + d(l_d,s_{k+1}))   \label{eq:dom_rule_2_proof}
\end{align}
Then, any complete path that involves $x_k$  belongs to the class of dominated solutions $D_2$ described by Lemma~\ref{prop:neighbor_items}.
\end{lemma}

\begin{proof}
First, note that the item $s_{k+1}$ of Lemma~\ref{lem:association_trans_path_2} corresponds to item $\hat{\pi}[i+1]$ in Lemma~\ref{prop:neighbor_items_2}, and that $r(s_{k+1})=r_{j}$ since $s_{k+1}\in o_{j}$. Given that $ S^{\text{batch}}=\{\} $ and $m^o\geq 1$ in state $\Theta_k$, item $s_{k+1}$ belongs to batch $\hat{B}_l$ that can be decomposed as follows: $\hat{B}_l=\hat{B}_l^I\cup \hat{B}_l^{II}$, where $\hat{B}_l^I$ is a set of orders for which all items have already been picked by the time state $\Theta_k$ is reached, and $\hat{B}_l^{II}$ is a set of completely unprocessed orders at state $\Theta_k$, one of which includes item $s_{k+1}$.

Using the same arguments as in the proof of Lemma~\ref{lem:association_trans_path_1}, each complete path of the state graph that involves transition $x_k\in X(\Theta_k)$ from state $\Theta_k$ with value $\Omega^{\text{*,cost}}(\Theta_k)$ completes the picking of item $s_k$ at time $C(s_k,\hat{\sigma})=\Omega(\Theta_k)$, thus property (\ref{eq:dom_rule_2_proof}) translates to property (\ref{eq:hilf0b}), and the path is associated with a weakly dominated solution in the set $D_2$. 
 \end{proof}
\subsubsection{Proof of Proposition~\ref{prop:dom_batch-closure}} \label{sec:proof_dom_batch_closure}

Again, the proof of Proposition~\ref{prop:dom_batch-closure} follows the same lines as the one presented in Section~\ref{sec:proof_dom_in_batch_1}.
Consider an instance $I$ for OBSRP-R. Lemma~\ref{prop:neighbor_batches} identifies a set $D_3$ of feasible solutions, that are weakly dominated by another feasible solution. 

\begin{lemma}~\label{prop:neighbor_batches}
Consider any feasible solution $\hat{\sigma}$ with a visiting sequence of picking locations $\hat{\pi}=(\hat{\pi}^I, \hat{\pi}^{II})$ and the respective sequence of batches $\hat{\pi}^{\text{batches}}, \vert \hat{\pi}^{\text{batches}}\vert =\hat{f}\in \mathbb{N}$ such that:
\begin{itemize}
\item 
The second subsequence $\hat{\pi}^{II}$ contains the batches $\hat{B}_{l}\cup\ldots\cup\hat{B}_{\hat{f}}$.
for some $l\in[\hat{f}]$. Let denote the
\textit{first location} of the second subsequence as $\hat{\pi}[i+1]$. 
\item There exists an order $o_{\underline{j}}\in \hat{B}_{l}\cup\ldots\cup\hat{B}_{\hat{f}}$, whose items are picked in the second subsequence $\hat{\pi}^{II}$ in $\hat{\sigma}$, which satisfies the following relation:
\begin{align}
r(\hat{\pi}[i+1]) \geq \max\{C^i, r_{\underline{j}} \} + \chi(o_{\underline{j}}) + \frac{1}{v} \cdot d(l_d,\hat{\pi}[i+1]), \label{eq:dom_batch_clos_proof}
\end{align}
\end{itemize}
where $C^i:= C(\hat{\pi}[i],\hat{\sigma})+ \frac{1}{v} \cdot d(\hat{\pi}[i],l_d)$ if $i\neq 0$ and $C^i:=0$ else; and $\chi(o_{\underline{j}}):=$ minimum time to pick order $o_{\underline{j}}$ in a single-order batch.

Then the following solution $\tilde{\sigma}$ is feasible and weakly dominates $\hat{\sigma}$:
\begin{itemize}
\item $\pi(\tilde{\sigma})=(\hat{\pi}^I, \pi^{\underline{j}}, \pi^{-\underline{j}})$ with $\pi^{\underline{j}}$ representing a visiting sequence of the picking locations in order $o_{\underline{j}}$ of travel distance $\chi(o_{\underline{j}})$, and $\pi^{-\underline{j}}|\hat{\pi}^{II}$ representing the subsequence of $\hat{\pi}^{II}$ after removing all the items of order $o_{\underline{j}}$.  In other words, solution $\tilde{\sigma}$ picks the items of order $o_{\underline{j}}$ directly after $\hat{\pi}^I$ and then resumes visiting the remaining locations in the same order as in $\hat{\sigma}$.
\item $\pi^{\text{batches}}(\tilde{\sigma})=(\hat{B}_{1},\ldots,\hat{B}_{l-1},\{o_{\underline{j}}\}, \hat{B}_{l}\setminus\{o_{\underline{j}}\}, \ldots \hat{B}_{\hat{f}}\setminus\{o_{\underline{j}}\}$),
i.e., in solution $\tilde{\sigma}$, order $o_{\underline{j}}$ is collected as a separate batch directly after $\hat{\pi}^I$.
\end{itemize}
\end{lemma}
\begin{proof}
By construction, $\tilde{\sigma}$ is a feasible solution, for instance, $\pi^{\text{batches}}(\tilde{\sigma})$ represents a mutually disjoint partition of the orders into batches and each batch contains at most $c$ orders. What remains to show, is that $z^{obj}(\tilde{\sigma})\leq z^{obj}(\hat{\sigma})$, where $obj\in\{$makespan, turnover\}, depending on the objective. 

Let compute the schedule of both solutions $\tilde{\sigma}$ and $\hat{\sigma}$ as described in (\ref{eq:pcart1})-(\ref{eq:pcart3}). 
Observe that the first $i$ items in $\tilde{\sigma}$ and $\hat{\sigma}$ are the same, thus all items and orders in $\hat{\pi}^I$ have the same completion times in both solutions. By definition, $C(o_{\underline{j}},\tilde{\sigma})\leq C(o_{\underline{j}},\hat{\sigma})$ . In solution $\tilde{\sigma}$, item $\hat{\pi}[i+1]$ is picked after all the items of order $o_{\underline{j}}$ are collected in some sequence $\pi^{\underline{j}}$.  Following the assumptions, and by applying (\ref{eq:pcart1})-(\ref{eq:pcart3}) the completion time of item $\hat{\pi}[i+1]$ in solution $\tilde{\sigma}$ is:
\begin{align}
C(\hat{\pi}[i+1], \tilde{\sigma}) \leq & \max\{\max\{C^i, r_{\underline{j}} \} + \chi(o_{\underline{j}})+ \frac{1}{v} d(l_d,\hat{\pi}[i+1]); r(\hat{\pi}[i+1])\}+t^p  \\
= & \ r(\hat{\pi}[i+1])+t^p,
\end{align}

The completion time of item $\hat{\pi}[i+1])$ in solution $\hat{\sigma}$ cannot be smaller by (\ref{eq:pcart1})-(\ref{eq:pcart3}),in other words:
\begin{align}
C(\hat{\pi}[i+1]), \tilde{\sigma})\leq C(\hat{\pi}[i+1]), \hat{\sigma})
\end{align}
After applying (\ref{eq:pcart3}) to compute the completion times of the remaining items and observing that the distances are metric and that the sequences of the visiting locations after picking $\hat{\pi}[i+1]$ coincide in $\tilde{\sigma}$ and $\hat{\sigma}$, but $\tilde{\sigma}$ excludes the locations of the items from the already picked order $o_{\underline{j}}$, we receive that all $C(s,\tilde{\sigma})\leq C(s,\hat{\sigma})$ and $C(o_j,\tilde{\sigma})\leq C(o_j,\hat{\sigma})$, for all items $s$ and orders $o_j$ completed after $\hat{\pi}[i+1]$ in $\tilde{\sigma}$. By definition of the objective functions, (\ref{eq:obj_pushcart_makespan}) and (\ref{eq:obj_turnover}), this completes the proof.
 \end{proof}

Let denote by $D\_3$ the set of weakly dominated solutions $\hat{\sigma}$ described by Lemma~\ref{prop:neighbor_batches}. Again, the following Lemma~\ref{lem:association_trans_path_3} uses the one-to-one correspondence between feasible solutions for an instance and complete paths in the corresponding state graph (see Proposition~\ref{prop:1-1correspondence}), to associate any transition $x_k$ from a label of state $\Theta_k$ with value $\Omega^{*,\text{cost}}$ described by Proposition~\ref{prop:dom_batch-closure} to a dominated solution in  $D_3$. Therefore, we must note that $\frac{1}{v}(2L+(a(o_{\underline{j}})+1)W)+\vert o_{\underline{j}} \vert \cdot t^p$, $a(o_{\underline{j}}):=$ aisle containing items from $o_{\underline{j}}$ is an \textit{upper bound} for $\chi(o_{\underline{j}})$, the the conditions described in Proposition~\ref{prop:dom_batch-closure} imply the ones of Lemma~\ref{lem:association_trans_path_3}.

\begin{lemma}~\label{lem:association_trans_path_3}
Consider a feasible transition $x_k\in X(\Theta_k)$ from a state $\Theta_k=(s_k,0,\{\},O^{\text{pend}})$ with value $\Omega^{\text{*,cost}}$ that dictates a next picking location $s_{k+1}\in o_{j}, o_{j} \in O^{\text{pend}}$ with the following property  
\begin{align}
r_{j}\geq \max\{ \Omega^{\text{*,cost}}(\Theta_k);  r_{\underline{j}}  \} +  \chi(o_{\underline{j}})+ \frac{1}{v} \cdot d(l_d,s_{k+1})      \label{eq:prop_proof}
\end{align}
for some $o_{\underline{j}}\in O^{\text{pend}}$. Then, any closed path in the state graph that involves this transition belongs to the set of dominated solutions $D_3$ described by Lemma~\ref{prop:neighbor_batches}.
\end{lemma}
\begin{proof}
Note that the item $s_{k+1}$ of Lemma~\ref{lem:association_trans_path_3} corresponds to item $\hat{\pi}[i+1]$ in Lemma~\ref{prop:neighbor_batches}. Since $m^o=0$ in $\Theta_k$, a new batch is initiated by the pick of $s_{k+1}$.
Using the same arguments as in the proof of Lemma~\ref{lem:association_trans_path_1}, each complete path that involves transition $x_k\in X(\Theta_k)$ of  $\Theta_k$ with value $\Omega^{\text{*,cost}}(\Theta_k)$ completes the picking of item $s_k$ at time $\Omega^{\text{*,cost}}(\Theta_k)=C^i-\frac{1}{v}\cdot d(s_k,l_d)$. Therefore, property (\ref{eq:prop_proof}) translates to property (\ref{eq:dom_batch_clos_proof}), and thus the path is associated with a weakly dominated solution in the class $D_3$. 
\end{proof}

\subsection{Proof of analytical properties} \label{sec:selection_rate}
\subsubsection{Proof of Lemma~\ref{lemma:strategic_reloc}} \label{sec:proof_limited_relocation_gain}

For simplicity, let consider a picking time of $t^p=0$. The proof for non-zero picking times follows the same logic.
Let $s=(s_1,s_2,...,s_l)$ be the visiting sequence of picking locations \textit{and the depot}, of CIOS and ALG, starting with the first picking location. Let $done^{\text{CIOS}}(s_k)$ and $done^{\text{ALG}}(s_k)$ denote the time when CIOS and ALG, respectively, \textit{finish} their task (visit for the depot and picking for a picking location) at location $s_k$, $k\in [l]$. Let $dep^{\text{ALG}}(s_k)$ be the time the picker in ALG \textit{departs} from location $s_k, k \in [l]$. Recall that $r(s_k)$ for $s_k\in S$ denotes the arrival time of picking location $s_k$. We extend this notation by writing $r(s_k)=0$ if $s_k=l_d$.

\textit{Induction over $k\in [l]$}. \\
\textit{Initialization; $k=1$}. For the first picking location $s_1$, we have: 
\begin{align}
   done^{\text{ALG}}(s_1)- done^{\text{CIOS}}(s_1) \leq r(s_1) + d(l_d,s_1)- r(s_1) \leq \max_{i,j\in S \cup \{l_d\}}\{d(i,j)\}
\end{align}
\textit{Inductive hypothesis.} We suppose that for fixed $k\in [l]$
\begin{align}
    done^{\text{ALG}}(s_k)- done^{\text{CIOS}}(s_k) \leq  \max_{i,j\in S \cup \{l_d\}}\{d(i,j)\}
\end{align}
\textit{Inductive step, $k \rightarrow k+1$.}\\
We distinguish two cases. 

Case 1: $dep^{\text{ALG}}(s_k)=done^{\text{ALG}}(s_k)$, i.e. the picker in ALG departs immediately towards $s_{k+1}$ after  she reaches and terminates at location $s_k$, since $r(s_{k+1})\leq done^{\text{ALG}}(s_k)$. Then:
\begin{align}
   &done^{\text{ALG}}(s_{k+1})- done^{\text{CIOS}}(s_{k+1}) \nonumber \\
   =  \quad & done^{\text{ALG}}(s_k) + d(s_{k},s_{k+1}) - \max\{r(s_{k+1}); done^{\text{CIOS}}(s_k) + d(s_{k},s_{k+1})\} \\
   \leq \quad & done^{\text{ALG}}(s_k) + d(s_{k},s_{k+1}) -  done^{\text{CIOS}}(s_k) - d(s_{k},s_{k+1}) \\
   \leq  \quad &\max_{i,j\in S \cup \{l_d \}}\{d(i,j)\}
\end{align}
by the inductive hypothesis.

Case 2: $dep^{\text{ALG}}(s_k)=r(s_{k+1})$, i.e. the picker in ALG must wait for the arrival of $s_{k+1}$ after  she terminates at location $s_k$. Then:
\begin{align}
   &done^{\text{ALG}}(s_{k+1})- done^{\text{CIOS}}(s_{k+1}) \nonumber \\
   =  \quad & r(s_{k+1}) + d(s_{k},s_{k+1}) - \max\{r(s_{k+1}); done^{\text{CIOS}}(s_k) + d(s_{k},s_{k+1})\} \\
   \leq \quad & r(s_{k+1}) + d(s_{k},s_{k+1}) - r(s_{k+1})  \\
   \leq  \quad &\max_{i,j\in S \cup \{l_d\}}\{d(i,j)\}
\end{align}

\subsection{Additional computational results} \label{sec:intervention}

The positive effect of intervention was demonstrated in Figure~\ref{fig:avg_gaps_small} and Table~\ref{tab:gaps} of Section~\ref{sec:good_online_algs} for the Reopt policy with a VTW waiting strategy for $N=2$ orders in the queue. We show in Table~\ref{tab:intervention_policies} that this effect is stable across a large range of policies. Specifically, we extend the study of the gap to the perfect anticipation result with- and without intervention to Reopt with no waiting, and the FIFO policy, which performs optimal routing, FIFO-batching, and no waiting. For both policies, intervention \textit{reduced all average} gaps with the \textit{makespan} objective by impressive 1.7-4.0 percentage points, with the \textit{turnover} objective the reduction was even 6.0-20.8 percentage points.

\begin{table}
\caption{Effect of intervention on average- (worst-)  gap (\%) to the perfect anticipation result for non-interventionist policies} \label{tab:intervention_policies}
\vspace{0.2cm}
\centering
\scriptsize{
\begin{tabular}{@{}r|rr|rr|rr|rr}
\toprule
& \multicolumn{4}{c|}{\underline{Makespan objective} \vspace{0.5mm}} & \multicolumn{4}{c}{\underline{Turnover objective} \vspace{0.5mm}} \\
 & \multicolumn{2}{c}{Reopt} & \multicolumn{2}{c|}{FIFO} & \multicolumn{2}{c}{Reopt} & \multicolumn{2}{c}{FIFO}  \\
Setting ($n^o=15$) & Gap N-interv. & interv. &  Gap N-interv. & interv. &  Gap N-interv. & interv. & Gap N-interv. & interv.\\
\midrule
 $LargeOrders\_c2\_r200$  &  6.3 (11.2) & -2.8 (-3.6) & 7.5 (14.8) & -2.0 (-1.5) & 37.7 (54.0) & -13.3 (-14.3) & 48.3 (95.9) & -15.5 (-46.9)\\ 
 $SmallOrders\_c2\_r200$  & 5.1 (15.2) & -2.1 (-6.2) & 6.6 (17.2) & -1.7 (-1.7) &  46.7 (75.7) & -12.5 (-21.1) & 63.1 (93.3) & -10.5 (-5.3)\\ 
  $SmallOrders\_c2\_r250$  & 8.0 (16.6) & -3.4 (-6.0) & 8.8 (17.4) & -2.3 (-1.0) & 47.3 (101.1) & -18.1 (-48.6) & 61.6 (106.1) & -13.1 (-23.0)\\ 
 $SmallOrders\_c4\_r250$  & 8.6 (20.2) & -3.8 (-6.4) & 9.4 (25.9)& -4.0 (-5.6) & 49.7 (93.2) & -20.8 (-38.3) & 64.1 (97.1) & -6.4 (-5.0)\\ 
 Overall  & 7.0 (20.2) & -3.0 (-6.5)  & 8.1 (25.9) & -2.5 (-5.6) & 45.3 (101.1)& -16.6 (-46.5) & 59.3 (106.1) & -6.0 (-14.0)\\ 
    \midrule   
    \multicolumn{9}{l}{Note. \textit{Columns 2 \& 6 (4\& 8)Gap N-interv.:} Gap gaps of Reopt (FIFO) policy without intervention.  }\\
    \multicolumn{9}{l}{Note. \textit{Columns 3 \& 7 (5\& 9) interv.:} Effect on the gap of allowing intervention in Reopt (FIFO)  }\\
\end{tabular}
}
\end{table}

\end{document}